\documentclass[12pt,a4paper]{amsart}
\usepackage{amsmath,amssymb,amsfonts, float}
\usepackage[all]{xy}
\usepackage{caption}
\usepackage{enumerate}
\usepackage{mathpazo}
\usepackage{a4wide}

\usepackage{bm}
\usepackage{graphicx}
\usepackage[breaklinks=true]{hyperref}

\usepackage{xcolor}
\usepackage{multicol}
\usepackage{tabularx}

\usepackage{hyperref}
\usepackage[capitalize]{cleveref}

\usepackage[normalem]{ulem}

\newtheorem{thm}{Theorem}[section] 
\newtheorem*{thm*}{Theorem} 
\newtheorem{prop}[thm]{Proposition}
\newtheorem{lem}[thm]{Lemma}
\newtheorem{cor}[thm]{Corollary}

\theoremstyle{definition}
\newtheorem{definition}[thm]{Definition}
\newtheorem{expl}[thm]{Example}

\newtheorem{question}[thm]{Question}
\newtheorem{rem}[thm]{Remark}

\DeclareMathOperator{\C}{\mathbb{C}}
\DeclareMathOperator{\Z}{\mathbb{Z}}
\DeclareMathOperator{\N}{\mathbb{N}}
\DeclareMathOperator{\F}{\mathbb{F}}

\DeclareMathOperator{\Hom}{{\rm Hom}}

\DeclareMathOperator{\Div}{{\rm Div}}

\DeclareMathOperator{\Tr}{{\rm Tr}}

    \DeclareFontFamily{U}{wncy}{}
    \DeclareFontShape{U}{wncy}{m}{n}{<->wncyr10}{}
    \DeclareSymbolFont{mcy}{U}{wncy}{m}{n}
    \DeclareMathSymbol{\Sha}{\mathord}{mcy}{"58}

\numberwithin{equation}{section}

\newcommand{\lcm}{\textrm{lcm}}

\DeclareSymbolFont{bbold}{U}{bbold}{m}{n}
\DeclareSymbolFontAlphabet{\mathbbold}{bbold}

\usepackage{hyperref}

\newcommand{\s}{\text{ss}}

\newcommand{\rad}{{\rm rad}}

\begin{document}
\sloppy

\title{On the gcd-graphs over polynomial rings }
 \author{ J\'an Min\'a\v{c}, Tung T. Nguyen, Nguy$\tilde{\text{\^{e}}}$n Duy T\^{a}n }
\address{Department of Mathematics, Western University, London, Ontario, Canada N6A 5B7}
\email{minac@uwo.ca}
\date{\today}

 \address{Department of Mathematics and Computer Science, Lake Forest College, Lake Forest, Illinois, USA}
 \email{tnguyen@lakeforest.edu}
 
  \address{
Faculty Mathematics and 	Informatics, Hanoi University of Science and Technology, 1 Dai Co Viet Road, Hanoi, Vietnam } 
\email{tan.nguyenduy@hust.edu.vn}

\thanks{JM is partially supported by the Natural Sciences and Engineering Research Council of Canada (NSERC) grant R0370A01. He gratefully acknowledges the Western University Faculty of Science Distinguished Professorship 2020-2021. TTN is partially supported by an AMS-Simons Travel Grant. Parts of this work were done while he was a postdoc associate at Western University. He thanks Western University for their hospitality. NDT is partially supported by the Vietnam National
Foundation for Science and Technology Development (NAFOSTED) under grant number 101.04-2023.21}

\begin{abstract}
Gcd-graphs over the ring of integers modulo $n$ are a natural generalization of unitary Cayley graphs. The study of these graphs has foundations in various mathematical fields, including number theory, ring theory, and representation theory. Using the theory of Ramanujan sums, it is known that these gcd-graphs have integral spectra; i.e., all their eigenvalues are integers. In this work, inspired by the analogy between number fields and function fields, we define and study gcd-graphs over polynomial rings with coefficients in finite fields. We establish some fundamental properties of these graphs, emphasizing their analogy to their counterparts over $\Z.$
\end{abstract}

\maketitle

\tableofcontents

\section{Introduction}
Let $n$ be a positive integer.  The unitary Cayley graph on the ring of integers modulo $n$ is defined as the Cayley graph $G_n = {\rm Cay}(\Z/n\Z, U_n)$ where $\Z/n\Z$ is the ring of integers modulo $n$ and $U_n = (\Z/n\Z)^{\times}$ its unit group. More specifically, $G_n$ is equipped with the following data:
\begin{enumerate}
    \item The vertex set of $G_n$ is $\Z/n\Z$;
    \item Two vertices $a, b \in V(G_n)$ are adjacent if and only if $a-b \in U_n$. 
\end{enumerate}
The unitary Cayley graph $G_n$ was first formally introduced in \cite{klotz2007some} even though we can trace it back to the work of Evans and Erdős \cite{erdos1989representations}. Due to its elegance and simplicity, the unitary Cayley graph has been further studied and generalized by various works in the literature. For example, \cite{unitary} studies the unitary Cayley graph of a finite commutative ring. In \cite{kiani2012unitary}, the authors generalize this study further to finite rings which are not necessarily commutative. 
In \cite{nguyen2024certain}, a subset of the authors studies various arithmetic and graph-theoretic properties of a subclass of $k$-unitary Cayley graphs as defined in \cite{podesta2021_finitefield, podesta2019spectral} (when $k=1$, these graphs are precisely unitary Cayley graphs). We refer the interested readers to  \cite{bavsic2015polynomials, chudnovsky2024prime, klotz2007some, nguyen2024certain} and the references therein for some further topics in this line of research.

As explained in \cite[Section 4]{klotz2007some}, a particularly intriguing arithmetical property of $G_n$ is that its spectrum can be described by the theory of circulant matrices and Ramanujan sums. A consequence of this fact is that the unitary Cayley graph has an integral spectrum; i.e., all of its eigenvalues are integers. 
In \cite{klotz2007some}, the authors note that the unitary Cayley graph is not the sole graph exhibiting an integral spectrum. They identify a closely related family of graphs sharing this characteristic, known as gcd-graphs which we now recall. Let $D =\{d_1, d_2, \ldots, d_k \}$ be a set of proper divisors of $n$. The gcd-graph $G_n(D)$ is the graph with the following data. 
\begin{enumerate}
    \item The vertex set of $G_n(D)$ is $\Z/n\Z.$
    \item Two vertices $a,b \in G_n$ are adjacent if and only if $\gcd(a-b, n) \in D.$
\end{enumerate}
We remark that the unitary Cayley graph $G_n$ is nothing but $G_{n}(\{1\}).$
We also note that  the gcd-graph $G_n(D)$  is loopless since $D$ does not contain $n$. Additionally, the graph $G_n(D)$ is undirected since  $\gcd(a-b,n)=\gcd(b-a,n)$. It is known that the spectrum of $G_n(D)$ is a summation of various Ramanujan sums (see \cite[Section 4]{klotz2007some}). In particular, all of its eigenvalues are integers. It turns out that, the converse is also true. 

\begin{thm}[{See \cite[Theorem 7.1]{so2006integral}}]
    An $\Z/n\Z$-circulant graph $G$ is an integral graph if and only if $G = G_n(D)$ for some $D.$
\end{thm}

Let $\F_q$ be a finite field. As observed by Artin in \cite{artin2} and Andr\'e Weil in \cite{weil1939analogie}, there is a strong analogy between the ring  $\Z$ of integers and the ring $\F_q[x]$ of polynomials with coefficients in $\F_q$ (and more generally, between number fields and function fields). Consequently, it is of reasonable interest to define and investigate gcd-graphs over $\mathbb{F}_q[x]$. Fortunately, the analogous definition of gcd-graphs over $\mathbb{F}_q[x]$ is relatively straightforward.  More specifically, let $f \in \F_q[x]$ be a non-zero element in $\F_q[x]$. Let $D = \{f_1, f_2, \ldots, f_{k} \}$ be a subset of the set of divisors  $\Div(f)$ of $f$; i.e, $f_i \mid f$ for all $1 \leq i \leq k.$ Let 
\[ S_{D} = \{g \in \F_q[x]/(f) : \gcd(g,f) \in D \}. \] 
\begin{rem}
   In general, the greatest common divisor is only defined up to associates. In our case, however, we can make a canonical choice for the great common divisor by requiring it to be a monic polynomial. Therefore, unless we explicitly state the contrary, we will assume throughout this article that all involved polynomials are monic. 
\end{rem}
\begin{definition}
     The \emph{gcd-graph} $G_{f}(D)$ is the Cayley graph $\Gamma(\F_q[x]/(f), S_{D}).$ More precisely, it is the graph equipped with the following data:
\begin{enumerate}
    \item The vertex set of $G_{f}(D)$ is the finite ring $\F_q[x]/(f)$,
    \item Two vertices $u, v$ are adjacent if and only if $u-v \in S_{D}$. In other words, ${\gcd(u-v, f) \in D}.$  
\end{enumerate}
\end{definition}

 Similar to the case of gcd-graphs over $\Z$, $G_f(D)$ is undirected and loopless. Furthermore, $G_f(D)$ is $|S_D|$-regular. The graph $G_f(\{1\})$, which is just the unitary Cayley graph over {$\F_q[x]/(f)$}, will play an important role in study of the connectedness and bipartiteness of general gcd-graphs over $\F_q[x]$. By \cite[Lemma 4.33]{chudnovsky2024prime}, $G_f(\{1\})$ is not connected if and only if $\F_q=\F_2$ and $x(x+1)\mid f$.

\begin{expl}
     \cref{fig:z_25} shows the graph $G_{f}(D)$ where $f = x(x+1) \in \F_{3}[x]$ and ${D = \{x, x+1 \}}$. {We have $S_D=\{x, 2x, x+1, 2(x+1)\}$. The graph $G_{f}(D)$}
     is a regular graph of degree $4.$
\begin{figure}[H]
\centering
\includegraphics[width=0.4 \linewidth]{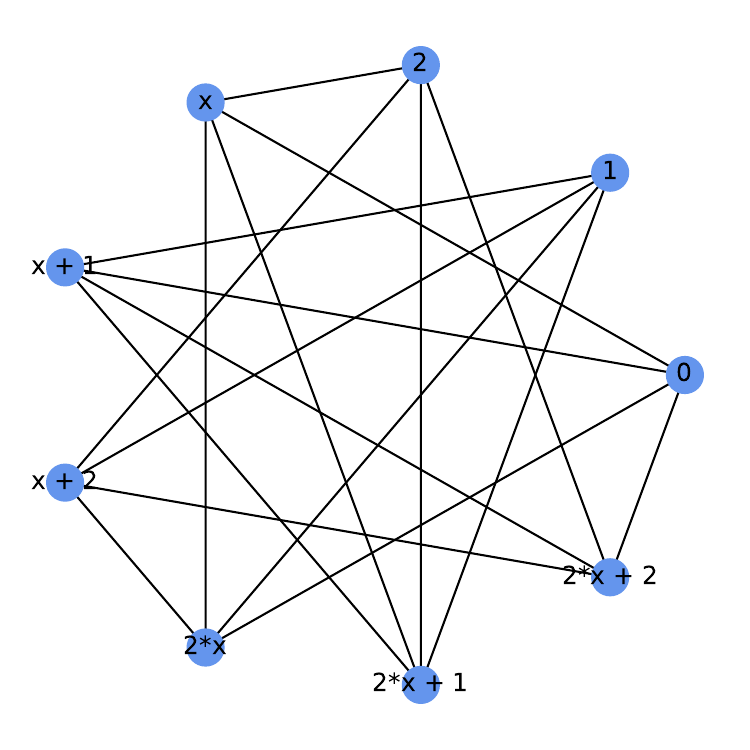}
\caption{The gcd-graph $G_{x(x+1)}(\{x, x+1 \})$}
\label{fig:z_25}
\end{figure}
\end{expl}

\subsection*{Outline}
In this article we introduce and study some foundational properties of gcd-graphs defined over \(\mathbb{F}_q[x]\). In \cref{sec:background}, we recall some standard definitions in graph theory that will be used throughout the article. \cref{sec:connectedness} investigates the question of when a gcd-graph is connected and anti-connected. (Recall that a graph is anti-connected if its complement is connected.) In \cref{sec:bipartite}, we provide the necessary and sufficient conditions for a gcd-graph to be bipartite. We note that our investigation provides additional insights into the properties of Cayley graphs. For example, in  recent works \cite{podesta2021_finitefield, podesta2025GP}, the authors examine the connectedness and bipartiteness of a related class of graphs known as generalized Paley graphs.  \cref{sec:prime} explores another graph-theoretic property of gcd-graphs: their primeness property. Here, we describe the necessary and sufficient conditions for a gcd-graph to be prime. We then apply this result to give an explicit criterion for a gcd-graph to be prime under some mild assumptions on the set $D$. \cref{sec:spec} examines the spectra of gcd-graphs using symmetric algebras and Ramanujan sums.  We find it quite remarkable that the explicit formulas for these spectra are identical to those that appeared in the number field case. In \cref{sec:perfect}, we provide some sufficient conditions for gcd-graphs to be perfect. Finally, in \cref{sec:induced}, we investigate whether a given graph can be realized as an induced subgraph of a gcd-graph. As a by-product, we prove an analogous result to a theorem of Erdős and Evans in the number field case. 

\subsection*{Code}
The code that we develop to generate gcd-graphs and do experiments on them can be found at \cite{Nguyen_gcd_graph}. 

\section{Background from graph theory} \label{sec:background}
In this section, we recall some basic concepts in graph theory that we will use throughout this article. 
\begin{definition}[The complete graph $K_n$]
The \emph{complete graph} $K_n$ is the graph on $n$ vertices which are pairwisely adjacent. 
\end{definition}
We now review two types of graph products that appear in our work, starting with the tensor product.

\begin{definition}[Tensor product of graphs] \label{def:tensor_product}
    Let $G,H$ be two graphs. The \emph{tensor product} $G\times H$ of $G$ and $H$ (also known as the direct product) is the graph with the following data:
    \begin{enumerate}
    \item The vertex set of $G \times H$ is the Cartesian product $V(G) \times V(H)$, 
    \item Two vertices $(g,h)$ and $(g',h')$ are adjacent in $G \times H$ if and only if $(g,g') \in E(G)$ and $(h, h') \in E(H).$
\end{enumerate}
\end{definition}
Next, we recall the definition of the wreath product of graphs.
\begin{definition}[Wreath product of graphs] \label{def:wreath_product}
  Let $G, H$ be two graphs. We define the {\emph{wreath product}}  of $G $ and $H$ as the graph $G * H$  with the following data:
  \begin{enumerate}
      \item The vertex set of $G * H$ is the Cartesian product $V(G ) \times V(H)$,
      \item $(g,h)$ and $(g',h')$ are adjacent in $G * H$ if either $(g,g') \in E(G)$ or $g=g'$ and $(h,h') \in E(H)$. 
  \end{enumerate}
\end{definition}

\begin{definition}[Graph morphisms] \label{def:graph_morphism}
 Let $G_1$ and $G_2$ be two graphs. We define a {\emph{graph morphism}} between $G_1$ and $G_2$ to be a map $f$ from $V(G_1)$ to $V(G_2)$ that preserves edges. More precisely, if $u,v \in V(G_1)$ are adjacent in $G_1$, then $f(u),f(v)$ are adjacent in $G_2$.
\end{definition}

Finally, we  recall the definition of induced subgraphs.
\begin{definition}[Induced subgraphs]
 Let $G_1$ and $G_2$ be two graphs. We say that $G_1$ is an {\emph{induced subgraph}} of $G_2$ if there exists a graph morphism $f\colon G_1 \to G_2$ such that the following conditions hold. 
 \begin{enumerate}
     \item The map $f\colon V(G_1) \to V(G_2)$ is injective. 
     \item For $u,v \in G_1$, $u,v$ are adjacent if and only if $f(u)$ and $f(v)$ are adjacent in $G_2.$
 \end{enumerate}
    
\end{definition}

\section{Connectedness of gcd-graphs} \label{sec:connectedness}

In this section, we will give the necessary and sufficient conditions for $G_{f}(D)$ to be connected. Unlike the number field case where these conditions are rather straightforward, the function field case is a bit more complicated. This is partially due to the fact that the unitary Cayley graph on $\F_q[x]/(f)$ is not always connected. We start our discussion with the following simple observation. 
\begin{lem}  \label{lem:connected}
If $G_f(D)$ is connected then $\gcd(f_1,\ldots,f_k)=1$.
    \end{lem}
    \begin{proof}
        Let $f_0=\gcd(f_1,\ldots,f_k)$. We observe that if two vertices $u$ and $v$ are adjacent then  $f_0$ divides $u-v$. 
        Now since  $G_f(D)$ is connected, there is a path connecting $0$ and $1$. This implies that $f_0$ divides $1-0=1$. Hence $\gcd(f_1,\ldots,f_k)=1$.
    \end{proof}
In the case of $\Z$, the converse of \cref{lem:connected} is true as well; i.e. the gcd-graph $G_{n}(D)$, where $D=\{d_1, d_2, \ldots, d_k \}$ is connected if and only if ${\rm gcd}(d_1, d_2, \ldots, d_{k})=1.$ In the case of $\F_q[x]$, this condition is not sufficient. For example, let $D=\{1 \}$, $\F_q=\F_2$, $f= x(x+1).$ Then $\F_q[x]/(f) \cong \F_2 \times \F_2.$ In this case 
\[ G_{f}(\{1\}) \cong K_2 \times K_2, \]
which is not connected (in general, this is the only obstruction where the unitary graph over a commutative ring fails to be connected, see \cite{chudnovsky2024prime} and \cref{prop:not_connected})). In the case of the gcd-graphs, we have the following result. 

\begin{thm} \label{prop:connected}
    $G_{f}(D)$ is connected if the following conditions hold.
    \begin{enumerate}
        \item $\gcd(f_1, f_2, \ldots, f_k)=1$.
        \item The unitary Cayley graph $G_{f}(\{1\})$ is connected. 
    \end{enumerate}
\end{thm}

\begin{proof} 
   Let $R= \F_q[x]/(f)$. We need to show that $R= \langle S_{D} \rangle$ the abelian group generated by $S_{D}$. Let $ a \in R.$ Because $\gcd(f_1, f_2, \ldots, f_k)=1$, we can find $a_1, a_2, \ldots, a_k \in R$ such that 
    \[ a = \sum_{i=1}^k a_i f_i.\]
    Since $G_{f}(\{1 \})$ is connected, for each $1 \leq i \leq k$, we can find write 
    \[ a_i = \sum_{j=1}^{n_i} m_{ij} s_{ij}, \]
    where $m_{ij} \in \Z$ and $s_{ij} \in R^{\times}.$ Consequently, we can write 
    \[ a = \sum_{i=1}^k \sum_{j=1}^{n_i} m_{ij} s_{ij} f_i. \]
    Since $f_i\mid f$, we see that $\gcd(s_{ij} f_i, f)= f_i \in D$. This shows that $a \in \langle S_{D} \rangle$. Since this is true for all $a$, we conclude that $R= \langle S_{D} \rangle$.
\end{proof}
We can classify all $f$ such that $G_{f}(\{1\})$ is not connected. 
\begin{prop}[{See \cite[Lemma 4.33]{chudnovsky2024prime}}] \label{prop:not_connected}
    $G_{f}(\{1 \})$ is not connected if and only if $\F_q = \F_2$ and $x(x+1) \mid f.$
\end{prop}

By \cref{prop:connected} and \cref{prop:not_connected}, we have the following corollary. 

\begin{cor} \label{cor:equiv_conneted}
    Suppose that one of the following conditions holds. 
    \begin{enumerate}
        \item $\F_q \neq \F_2$.
        \item $\F_q =\F_2$ and $x(x+1) \nmid f.$
    \end{enumerate}
Then $G_{f}(D)$ is connected if and only if $\gcd(f_1, f_2, \ldots, f_k)=1.$
\end{cor}

We now deal with the case where $G_{f}(\{1\})$ is not connected. By \cref{prop:not_connected}, this implies that $\F_q = \F_2$ and $x(x+1) \mid f.$

\begin{lem} \label{lem:gcd_lem}
    Let $g$ be a divisor of $f$. Then for every polynomial $h \in \F_q[x]$ 
    \[ \gcd(h, g) = \gcd(\gcd(h, f), g).\]
\end{lem}

\begin{proof}
    Let $m = \gcd(h,f).$ We need to show that  $\gcd(h,g) = \gcd(m,g).$ 
    We first claim that $\gcd(h,g) \mid \gcd(m,g).$ By definition, $\gcd(h,g) \mid g.$ We also have ${\gcd(h,g) \mid \gcd(h,f)=m}.$ Therefore, $\gcd(h,g) \mid \gcd(m,g).$

    Conversely, we claim that $\gcd(m,g) \mid \gcd(h,g).$ This is clear since $m\mid h$.

    In summary, we have $\gcd(m,g) \mid \gcd(h,g)$ and $\gcd(h,g) \mid \gcd(m,g).$ This shows that $\gcd(m,g)=\gcd(h,g).$
\end{proof}

\begin{prop} \label{prop:projection_morphism}
    Let $g$ be a divisor of $f$ and $\Phi_{f,g}\colon \F_q[x]/(f) \to \F_q[x]/g$ be the canonical projection map. Let 
    \[ \overline{D} = \{\gcd(f_i,g) : 1 \leq i \leq k \}. \]
    Then  $\Phi_{f,g}(S_{D}) \subseteq S_{\overline{D}}$. Consequently, $\Phi_{f,g}\colon G_{f}(D) \to G_{g}(\overline{D})$ is a graph morphism. 
\end{prop}

\begin{proof}
    Suppose that $a \in S_{D}.$ Then, there exists $i$ such that $\gcd(a,f)=f_i.$ By \cref{lem:gcd_lem}, we know that 
    \[ \gcd(a,g) = \gcd(\gcd(a,f), g) = \gcd(f_i,g).\] 
    This shows that $a \in S_{\overline{D}}.$ We conclude that $\Phi_{f,g}(S_{D}) \subseteq S_{\overline{D}}$.
\end{proof}

\begin{cor} \label{cor:connected_projection}
    If $G_{f}(D)$ is connected then $G_{g}(\overline{D})$ is connected as well. 
\end{cor}

\begin{proof}
    Let $\bar{a},\bar{b}$ be two vertices in $G_{g}(\overline{D}).$ Since $\Phi_{f,g}$ is surjective, we can find ${a, b \in \F_q[x]/(f)}$ such that $\Phi_{f,g}(a)=\bar{a}, \Phi_{f,g}(b) = \bar{b}.$ Since $G_{f}(D)$ is connected, there is a path $P$ from $a$ to $b.$ By \cref{prop:projection_morphism}, $\Phi_{f,g}(P)$ is a path from $\bar{a}$ to $\bar{b}.$ This shows that $G_{g}(\overline{D})$ is connected as well. 
\end{proof}

We are now ready to state our theorem in the case $G_{f}(\{1\})$ is not connected. 

\begin{thm} \label{thm:connected_hard}
    Suppose that $\F_q=\F_2$ and $x(x+1) \mid f.$ Let $$\Phi\colon \F_q[x]/(f) \to \F_q[x]/(x(x+1))$$ be the canonical projection map and $\bar{D}$ as described above. Then $G_{f}(D)$ is connected if and only if 
    \begin{enumerate}
        \item $\gcd(d_1, d_2, \ldots, d_k)=1$, and
        \item  the graph $G_{x(x+1)}(\bar{D})$ is connected. This condition is equivalent to $$|\overline{D} \setminus \{x(x+1) \}| \geq 2.$$
    \end{enumerate} 
\end{thm}
\begin{proof}
By \cref{cor:connected_projection}, we know that if $G_{f}(D)$ is connected then $(1)$ and $(2)$ holds. Conversely, let us assume that $(1)$ and $(2)$ both hold. We will show that $G_{f}(D)$ is connected. This is equivalent to showing that $R: = \F_q[x]/(f) = \langle S_D \rangle.$

The key idea of this proof is similar to the proof for \cref{prop:connected}. The main difficulty is to deal with the fact that $G_{f}(\{1 \})$ is not connected in this case. We will do this step by step.

We first claim that if $g \in R$ such that $x(x+1) \mid g$ then $g = s_1 +s_2$ where $s_1, s_2 \in R^{\times}.$ For an element $s \in R$, $s$ is a unit if and only $\Phi^{\s}(s) \in (R^{\s})^{\times}$ where $R^{\s}= R/J(R)$ is the semisimplification of $R$ (see \cite[Proposition 4.30]{chudnovsky2024prime}). Therefore, for this statement, we can assume that $R = R^{\s}$; namely $f$ is a squarefree polynomial. If we write $f = x(x+1)h$ where $\gcd(x(x+1),h)=1$ then we have the isomorphism 
\[ R \cong \F_2[x]/(x(x+1)) \times \F_2[x]/(h). \]
Under this isomorphism, $g$ is sent to $(0, g_1)$ where $g_1 \in \F_2[x]/(h).$ Since $F_2[x]/(h)$ is a product of fields of order bigger than $2$, every element in it can be written as the sum of two units; say $g_1 = t_1+t_2$ where $t_1, t_2 \in (\F_2[x]/(f))^{\times}.$ We then see that $g =s_1+s_2$ where $s_1= (1,t_1)$ and $s_2 = (1, t_2)$. By definition $s_1, s_2 \in R^{\times}.$

We now claim that if $g \in R$ such that $x(x+1) \mid g$ then $g \in \langle S_{D} \rangle.$ Since $\gcd(f_1,f_2,\ldots, f_k)=1$, we can find $a_1, a_2, \ldots, a_k\in R$ such that 
\[ a_1f_1+a_2f_2+\cdots+ a_kf_k=1.\]
Multiplying both sides with $g$, we get $g = \sum_{i=1}^k a_ig f_i.$ Since $x(x+1) \mid a_ig$, we can write $a_i g = s_{1i}+s_{2i}$ where $s_{1i}, s_{2i} \in R^{\times}.$ This shows that 
\begin{equation} \label{eq:sum_of_units}
 g = \sum_{i=1}^k (s_{1i}f_i +s_{2i}f_i).
\end{equation}
Since $s_{1i}, s_{2i} \in R^{\times}$, $s_{1i}f_i, s_{2i}f_i \in S_{D}.$ This shows that $g \in \langle S_{D} \rangle.$

Finally, let $g$ now be an arbitrary element in $R$. We claim that $g \in \langle S_D \rangle.$ By our assumption, the graph $G_{x(x+1)}(S_{\bar{D}})$ is connected, $0$ and $\Phi_{f,x(x+1)}(g)$ are connected by a path. Consequently,  we can write 
\[ g \equiv  \sum_{i} n_i \gcd(h_i, x(x+1)) \pmod{x(x+1)}, \]
where $n_i \in \Z$ and $\gcd(h_i, f) \in D$. We can check that over $\F_2[x]$ 
\[ \gcd(h, x(x+1)) \equiv  h \pmod{x(x+1)},\]
for all $h \in \F_2[x].$ Therefore, we can write 
\[ g \equiv \sum_{i} n_i h_i \pmod{x(x+1)}. \]
By the previous case, we know that $g- \sum_{i=1}^k n_i h_i \in \langle S_D \rangle.$ This shows that $g \in \langle S_D \rangle$ as well. 
Since this is true for all $g$, we conclude that $R = \langle S_D \rangle.$    
\end{proof}

\begin{rem}
    The proof of \cref{thm:connected_hard} also implies that for the case $\F_2[x]$ and ${x(x+1) \mid f}$, the map $\Phi_{f,g}\colon \F_2[x]/(f) \to \F_2[x]/(x(x+1))$ has the property that $a$ and $b$ belong to the same connected component in $ G_{f}(D)$ if and only $\Phi_{f,x(x+1)}(a)$ and $\Phi_{f,x(x+1)}(b)$ belong to the same connected component in $G_{x(x+1)}(\overline{D}).$ In fact, by \cref{cor:connected_projection}, if $a$ and $b$ are connected by a path then $\Phi_{f,x(x+1)}(a)$ and $\Phi_{f,x(x+1)}(b)$ are connected by a path. Conversely, if $\Phi_{f,x(x+1)}(a)$ and $\Phi_{f,x(x+1)}(b)$ are connected by a path in $G_{x(x+1)}(\overline{D})$, we can write 
\[ a-b \equiv \sum_{i} n_i h_i \pmod{x(x+1)}. \]
where $n_i \in \Z$ and $\gcd(h_1, f) \in D$. From the proof of \cref{thm:connected_hard}, we know that $x(x+1)R \subseteq \langle S_D \rangle$. Therefore, we conclude that $a-b \in \langle S_D \rangle.$ By definition, $a$ and $b$ are connected by a path in $G_{f}(D)$.

We conclude that $G_{f}(D)$ and $G_{x(x+1)}(\overline{D})$ has the same number of connected component. In particular, the number of connected components in $G_{f}(D)$ is at most $2.$

\end{rem}
{
\begin{expl}
     \cref{fig:z_26} shows the graph $G_{f}(D)$ where $f = x^2(x+1) \in \F_{2}[x]$ and ${D = \{x, x+1 \}}$. {We have $S_D=\{x, x+1, x^2+1\}$. The graph $G_{f}(D)$} is a regular graph of degree $3.$ Furthermore, by \cref{thm:connected_hard}, it is connected. 
\begin{figure}[H]
\centering
\includegraphics[width=0.4 \linewidth]{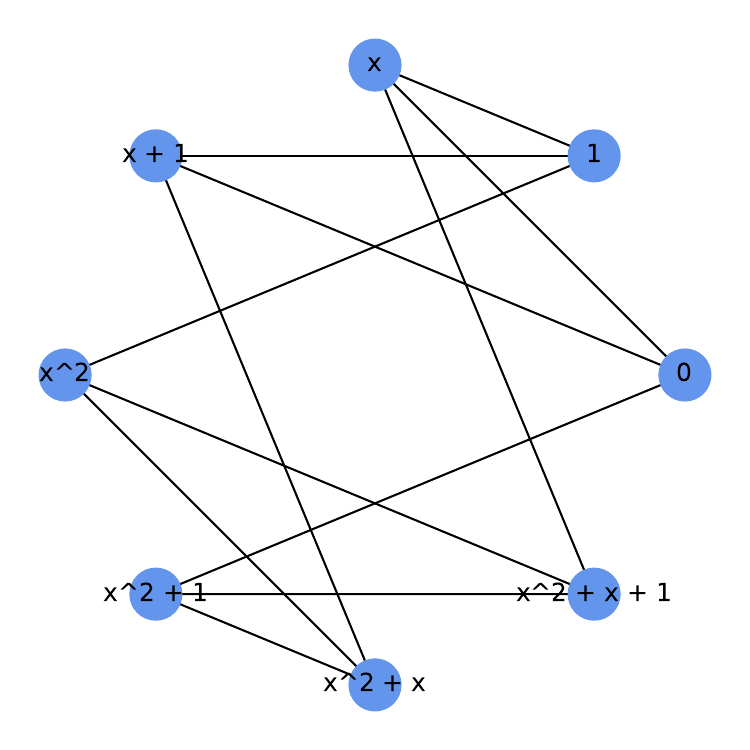}
\caption{The gcd-Cayley graph $G_{x^2(x+1)}(\{x, x+1 \})$}
\label{fig:z_26}
\end{figure}
\end{expl} }

\begin{rem} \label{rem:anti_connected}

While most of our discussions in this section concern the connectedness of $G_{f}(D)$, similar statements hold for the anti-connectedness of $G_{f}(D)$ as well. In fact, the complement of $G_f(D)$ is precisely $G_f(\Div(f) \setminus (D \cup \{f \})$ where $\Div(f)$ is the set of all proper divisors of $f.$
    
\end{rem}
\section{Bipartite property} \label{sec:bipartite}

A graph $G = (V,E)$ is called a {\emph{bipartite}} graph if  $V$ can be partitioned into two disjoint sets $V_1, V_2$ such that for every edge $(u,v) \in E$ either $u \in V_1, v \in V_2$ or $u \in V_2, v \in V_1$. In this case, we write $V = V_1 \sqcup V_2$ and we call $V_1, V_2$ independent sets in $G.$

Bipartite graphs model various real-life situations such as job assignments, resource allocation, stable marriage matching, gene-disease association, and much more. Additionally, from a theoretical point of view, bipartite graphs often provide a good test for theorems and algorithms on graphs. We refer the reader to \cite{bipartite} for some further discussions on this topic.

The goal of this section is to classify all $G_{f}(D)$ which is bipartite. We remark that while our focus is on gcd-graphs over $\F_q[x]$, everything we discuss in this section applies to gcd-graphs over $\Z$ as well (see \cref{rem:bipartite}). To the best of our knowledge, even over $\Z$, this topic has not been explored in the literature. Therefore, in addition to studying this problem for its own merits, we also hope to address a gap in the literature.

We start with the following lemma.

\begin{lem} \label{lem:connected_bipartite}
    Suppose that $G_{f}(D)$ is connected and bipartite. 
    
    \begin{enumerate}
    \item  Let  
    \[ I = \left\{ \sum_{i} n_i h_i : n_i \in \Z, h_i \in S_D \text{ and} \quad  \sum_{i=1}^k n_i \equiv 0 \pmod 2 \right\},\] 
    \[ I_1 = \left\{ \sum_{i} n_i h_i : n_i \in \Z, h_i \in S_D \text{ and} \quad  \sum_{i=1}^k n_i \equiv 1 \pmod 2 \right \}. \]
    Then $I$ is an subgroup of index $2$ in $(R,+).$ Furthermore, $I$ and $I_1$ are independent set such that  $V(G_{f}(D))= I \bigsqcup I_1.$ 
    \item If $G_{f}(\{1 \})$ is connected, then $I$ is an ideal in $R$ as well. 
    \item If $G_{f}(\{1\})$ is not connected, which is equivalent to $x(x+1) \mid f$ and $\F_q=\F_2$ by \cref{prop:not_connected}, then $x(x+1)R \subseteq I.$
    \end{enumerate}
\end{lem}
\begin{proof}
(1)  Let us consider the first part of this lemma. Because $G_{f}(D)$ is bipartite, we can write 
\[ V(G_{f}(D)) = A \bigsqcup B \]
where $A,B$ are two independent sets in $G_{f}(D).$ Without loss of generality, we can assume that $0 \in A$. We claim that $A=I$. By the proof of \cite[Proposition 2.6]{biswas2019cheeger}, we know that $A$ is a subgroup of $(R,+)$ of index $2.$ By definition, if $h_i \in S_D$, then $(0,h_i) \in E(G_{f}(D))$. Because $A$ and $B$ are disjoint independent set in $G_{f}(D)$, $h_{i} \in B$. Furthermore, since $A$ is a subgroup of index $2$ in $R$, $h_i + h_j \in A$ for all $h_i,h_j \in S_D.$  Consequently, we must have $I \subseteq A.$ 

By our assumption that $G_{f}(D)$ is connected we know that $R = \langle S_D \rangle$ and hence ${I \sqcup I_1 = \langle S_D \rangle = R}.$ Clearly if $a\in R\setminus I=I_1$ then $2a\in I$. 
From this, we can see that $I$ is a subgroup whose index is at most $2$ in $R$. Additionally, because $I \subseteq A$ and $A$ has index $2$ in $R$, we must have that $I=A$ and $I_1 = B.$

\smallskip
\noindent (2) Suppose that $G_{f}(\{1\})$ is connected. We will show that $I$ is an ideal in $R$ as well. The idea is similar to the proof of \cref{prop:connected} so we will be brief. Specifically, we note that for each $s \in R^{\times}$, $sI = I.$ Since $\langle R^{\times} \rangle =R$, this shows that $aI \subseteq I$ for all $a \in R.$ We conclude that $I$ is an ideal in $R$. 

\smallskip
\noindent(3) Finally, part $(3)$ follows from Equation \cref{eq:sum_of_units}.
\end{proof}

\begin{cor}
    Suppose that either $\F_q \neq \F_2$ or $\gcd(x(x+1),f)=1$. Then $G_{f}(D)$ is not a bipartite graph. 
\end{cor}
\begin{proof}
   If $\F_q \neq \F_2$ or $\gcd(x(x+1),f)=1$, we know that $G_{f}(\{1\})$ is connected by \cref{prop:not_connected}. Furthermore, $\F_q[x]/(f)$ has no ideal of index $2.$ By \cref{lem:connected_bipartite}, we conclude that $G_{f}(D)$ is not bipartite. 
\end{proof}
We now consider the case $\F_q=\F_2$ and $\gcd(x(x+1),f) \neq 1.$ We first consider the following easier case. 

\begin{thm}
    Suppose that $f$ is a polynomial in $\F_2[x]$ such that ${\gcd(x(x+1),f) \not \in \{1, x(x+1)\}}.$ Then $G_{f}(D)$ is bipartite if and only if $\gcd(x(x+1),f) \nmid f_i$ for all $1 \leq i \leq k.$
\end{thm}

\begin{proof}
    Without loss of generality, we can assume that $x \mid f$ but $x+1 \nmid f.$ First, let us assume that $x \nmid f_i$ for all $1 \leq i \leq k$. In this case, we can see that $A = xR$ and $B = 1+xR$ are independent subsets in $G_{f}(D)$ such that $V(G_{f}(D)) = A \bigsqcup B.$ This shows that $G_{f}(D)$ is bipartite. 

    Conversely, let us assume that $G_{f}(D)$ is bipartite. We claim that $x \nmid f_i$ for all ${1 \leq i \leq k}$. Since $x(x+1) \nmid f$, by \cref{prop:connected} and \cref{prop:not_connected}, we know that $G_{f}(\{1\})$ is connected. Therefore, by \cref{lem:connected_bipartite}, there exists an ideal $I$ of index $2$ in $R$ such that $I$ is an independent set in $G_{f}(D).$ We remark  that since the only ideal of index $2$ in $R$ is $xR$, $I=xR$. Because $0 \in I$ and $I$ is independent we must have 
    \[ xR \cap \{f_1, f_2, \ldots, f_k \} = \emptyset. \]
    In other words, $x \nmid f_i$ for all $1 \leq i \leq k.$
\end{proof}

Finally, let us consider the trickiest case where $x(x+1) \mid f.$

\begin{thm}
Suppose that $x(x+1) \mid f$ and that $G_{f}(D)$ is connected. Let 
    \[ \overline{D} = \{ \gcd(f_i, x(x+1)) : f_i \in D \}.\]
    Then $G_{f}(D)$ is bipartite if and only if $|\overline{D}| = 2.$
\end{thm}
\begin{proof}
Suppose that $|\overline{D}| = 2$. We claim that $G_{f}(D)$ is bipartite. Because $G_{f}(D)$ is connected, we know by \cref{thm:connected_hard} that $|\overline{D} \setminus \{x(x+1)\}| \geq 2.$ This shows that $\overline{D} \neq \{1, x(x+1) \}.$ We can check that $\overline{D}$ must be one of the following sets $\{1, x\}$, $\{1, x+1 \}$, $\{x, x+1 \}.$

    First, we consider the case that $\overline{D} = \{1, x+1 \}$. In this case $V(G_{f}(D)) = xR \bigsqcup (1+xR)$ is a decomposition of $G_{f}(D)$ into a disjoint union of two independent sets. Similarly, if $\overline{D}=\{1, x\}$ then $V(G_{f}(D)) = (x+1)R \bigsqcup (1+(x+1)R)$ is a decomposition of $G_{f}(D)$ into a disjoint union of two independent sets. Now, suppose that $\overline{D} = \{x, x+1 \}.$ 
    Let 
\[ A = \{g \in R| g(0) = g(1)\}, \]
and 
\[ B = \{g \in R : g(0) \neq g(1) \} = x + A.\]
We can check that if $a_1, a_2 \in A$ then $\gcd(a_1-a_2, x(x+1)) \in \{1, x(x+1) \}.$  By definition, $(a_1, a_2) \not \in E(G_{f}(D)).$ This shows that $A$ is an independent set in $G_{f}(D)$. Similarly, $B$ is an independent set in $G_{f}(D)$ as well. We conclude that $G_{f}(D)$ is bipartite.

Conversely, suppose $G_{f}(D)$ is bipartite. We claim  that $|\overline{D}| =2.$ By \cref{lem:connected_bipartite}, the subgroup 
     \[ I = \left\{ \sum_{i} n_i h_i : n_i \in \Z, h_i \in S_D \text{ and} \quad  \sum_{i=1} n_i \equiv 0 \pmod 2 \right\},\] 
is an independent set in $G_{f}(D).$ Furthermore, by part $(3)$ of \cref{cor:connected_projection}, we know that $x(x+1)R \subset I.$ Because $0$ is not connected to any node in $I$, we must have ${x(x+1)R \cap S_{D} = \emptyset}$. Consequently $x(x+1) \not \in \overline{D}.$ Suppose to the contrary that $|\overline{D}| \geq 3$. We then have $\overline{D} = \{1, x, x+1 \}.$ This implies that $G_{x(x+1)}(\overline{D})$ is the complete graph $K_4.$ In $K_4$, there is a path of length $2$ from $0$ to any other vertices. From this property, we conclude that for each $g \in R$, we can write 
\[ g \equiv \gcd(h_1, x(x+1)) + \gcd(h_2, x(x+1)) \pmod{x(x+1)}, \]
for some $h_1, h_2 \in S_D$. We then conclude that $g-h_1-h_2 \in x(x+1)R \subset I$. {By the definition of $I$}, $h_1+h_2 \in I$, we conclude that $g \in I $ as well. This shows that $I=R$, which is a contradiction. 
\end{proof}

\begin{rem} \label{rem:bipartite}
    By an almost identical argument, we can show that the gcd-graph $G_{n}(D)$ over $\Z$ with $D = \{d_1, d_2, \ldots, d_k\}$ and $\gcd(d_1, d_2, \ldots, d_k)=1$, is bipartite if and only if $2 \mid n$ and $2 \nmid d_i$ for all $1 \leq i \leq k.$
\end{rem}
\section{Primeness of gcd-graphs} \label{sec:prime}
For a given graph $G$, a {\emph{homogeneous}} set in  $G$ is a set $X$ of vertices of $G$ such that every vertex in $V(G) \setminus X$ is adjacent to either all or none of the vertices in $X$. A homogenous set $X$ is said to be non-trivial if $2 \leq X < |V(G)|$. The graph $G$ is said to be {\emph{prime}} if it does not contain any non-trivial homogeneous sets.

The concept of a homogeneous set appears in various branches of mathematics (see \cite{brandstadt1999graph, chudnovsky2024prime,  gallai1967transitiv, nguyen2022broadcasting} for some concrete examples). One of our main motivations comes from the fact that homogeneous sets allow us to decompose a network into a multilevel network of smaller graphs. From both theoretical and computational perspectives, such a decomposition is crucial for understanding the dynamics of multilevel networks (see \cite{boccaletti2014structure, jain2023composed, kivela2014multilayer, nguyen2022broadcasting}). In \cite{chudnovsky2024prime}, in collaboration with some other graph theorists, we completely classify prime unitary Cayley graphs on finite commutative rings. In this section, we study the problem for the gcd-graph $G_{f}(D).$ We remark that for a graph $G$, a connected component of $G$ (or of its complement) is a homogeneous set in $G$. Therefore, while studying the primeness of $G$, it is safe to assume that $G$ is connected and anti-connected. For this reason, we will assume throughout this section that $G_f(D)$ is both connected and anti-connected (we refer the reader to \cref{sec:connectedness} for precise conditions for these properties to hold). 

An important property of homogeneous sets is that they are preserved under a graph isomorphism. For this reason, we start our discussion with the following observation.

\begin{prop} \label{prop:mult_by_unit}
Let $a \in (\F_q[x]/(f))^{\times}.$ Let $m_a\colon \F_q[x]/(f) \to \F_q[x]/(f)$ be map induced by the multiplication by $a$. Then $m_a$ induces an automorphism on $G_{f}(D).$
\end{prop}
\begin{proof}
    Since $a \in (\F_q[x]/(f))^{\times}$, $m_a$ is an automorphism of $(\F_q[x]/(f), +)$.  Furthermore, $m_a$ preserves $S_D$; i.e $a S_{D} = S_D.$ As a result, $m_a$ is an automorphism of $G_{f}(D).$
\end{proof}

\begin{prop} \label{prop:prime_implies_ideal}
    Assume that $G_{f}(\{1\})$ is connected. Then, the following conditions are equivalent. 

    \begin{enumerate}
        \item $G_{f}(D)$ is not a prime graph. 
        \item There exists a non-trivial ideal $I$ in $\F_q[x]/(f)$ such that $I$ is a homogeneous set in $G_{f}(D).$
    \end{enumerate}
\end{prop}

\begin{proof}
    Clearly $(2)$ implies $(1)$. Let us now prove $(1)$ implies $(2)$. The proof that we discuss here is quite similar to the one that we gave for \cite[Theorem 4.1]{chudnovsky2024prime} and \cref{prop:connected}. Because $G_f(D)$ is not prime, we can find a maximal homogenous set $H$ containing $0.$ By \cite[Theorem 3.4]{chudnovsky2024prime}, we know that $H$ is a subgroup of $\F_q[x]/(f)$. We claim that it is an ideal as well. By \cref{prop:mult_by_unit}, we conclude that $aH = H$ for all $a \in (\F_q[x]/(f))^{\times}.$ Furthermore, since $G_{f}(\{1\})$ is connected, $aH \subseteq H$ for all $a \in \F_q[x]/(f)$ as well. This shows that $H$ is an ideal in $\F_q[x]/(f).$
\end{proof}

By \cref{prop:prime_implies_ideal}, in order to study the prime property of $G_f(D)$, it is essentially equivalent to classify $g$ such that the ideal generated by $g$ is a homogenous set in $G_f(D).$ In order to do so, we first introduce the following lemmas. 

\begin{lem} \label{lem:elementary}
    Let $a,b, c \in \F_q[x]$ such that $\gcd(a,b) = \gcd(b,c)= m$ and $c \mid f.$ Then, there exists $t \in \F_q[x]$ such that 
    \[ \gcd(a-bt, f)=c.\]
\end{lem}

\begin{proof}
By replacing $a,b,c,f$ by $\frac{a}{m}, \frac{b}{m}, \frac{c}{m}, \frac{f}{m}$, we can assume that $m=1.$    Because $\gcd(b,c)=1$, we can find $t_1$ such that $c \mid a-bt_1.$ Let us write $a-bt_1 = a_1c.$ 
    We will look for $t = t_1 +c t_2$ such that the condition $\gcd(a-bt, f)=c$ holds. By our choices of $t_1, t_2$, this is equivalent to
    \[ c= \gcd(a-b(t_1+ct_2),f) = \gcd(a_1c - b t_2 c, f) = c \gcd(a_1-b t_2, \frac{f}{c}). \]
    We remark that the relation $a-bt_1 = a_1 c$ and the fact that $\gcd(a,b)=1$ imply that $\gcd(a_1, b)=1.$ By the Chinese remainder theorem, we can find $t_2$ such that  ${\gcd(a_1 - b t_2, \frac{f}{c})=1}$. 
    \end{proof}

\begin{lem} \label{lem:induced_graph_ideal}
Let $g$ be a divisor of $f$. Let $I_g$ be the ideal in $R=\F_q[x]/(f)$ generated by $g.$ The induced graph on $I_g$ is isomorphic to $G_{f/g}(D_g)$ where 
\[ D_g = \left \{\frac{f_i}{g} : f_i \in D, g \mid f_i \right \} .\]
\end{lem}

\begin{proof}
Every element in $I_g$ can be written in the form $gm$ for a unique $m \in \F_q[x]/(f/g)$.   Therefore we have a natural map $I_g \to \F_q[x]/(f/g)$ sending $gm \mapsto m.$ Furthermore,   for two elements $ga, gb \in I_g$, we have 
    \[ \gcd(ga-gb, f) = g \gcd(a-b, \tfrac{f}{g}). \]
    Therefore, we see that $\gcd(ga-gb, f) \in D$ if and only $\gcd(a-b, \frac{f}{g}) \in D_g.$ From this, we conclude that the induced graph on $I_g$ is naturally isomorphic to $G_{f/g}(D_g)$.
\end{proof}

\begin{thm} \label{thm:prime_criterior}
    Let $g \mid f$ be a divisor of $f$ and $I$ the ideal in $\F_q[x]/(f)$ generated by $g.$ Let 
    \[ D_1 = \left\{f_i \in D:  g \nmid f_i  \right \}, \quad  D_2 = \left\{f_i \in D : g \mid f_i \right \}.\]
    As in \cref{prop:projection_morphism} let 
    \[ \overline{D_1} = \left \{\gcd(f_i,g) : f_i \in D_1 \right \}, \quad \widetilde{D_2} = \left  \{\frac{f_i}{g} :  f_i \in D_2 \right \}.\]
    Then, the following statements are equivalent. 
    \begin{enumerate}
    
    \item $I$ is a homogeneous set in $G_f(D)$.
    \item $\Phi_{f,g}^{-1}(\overline{D_1}) \cap \Div(f)= D_1$ where $\Phi_{f,g}\colon \F_q[x]/(f) \to \F_q[x]/(g)$ is the canonical projection map. 
    \end{enumerate}
   
Furthermore, if the above equivalent conditions holds, $G_f(D)$ is isomorphic to the wreath product $G_{g}(\overline{D_1}) * G_{f/g}(\widetilde{D_2})$ {(see \cref{def:wreath_product} for the definition of wreath product)}.
\end{thm}

\begin{proof}
First, let us show that $(2)$ implies $(1)$. In other words, suppose that $g$ satisfies the condition that $\Phi_{f,g}^{-1}(\overline{D_1}) \cap \Div(f)= D_1$. We claim that the ideal $I$ generated by $g$ is a homogeneous set in $G_f(D).$ In fact, let $a \not \in I$ be an arbitrary element in $\F_q[x]/(f)$ and suppose that $a$ is adjacent to an element in $I$. By a translation, we can assume that $a$ is adjacent to $0$ in $G_f(D).$ We claim that $a$ is adjacent to all elements in $I$ as well. Let $gt$ be an element in $I$. We need to show that 
    \[ \gcd(a-gt, f) \in D.\]
Since $(a, 0) \in E(G_f(D))$ we know that $\gcd(a,f) \in D.$ Because $a \not \in I$, we know further that $\gcd(a,f) \in D_1$. Additionally, by \cref{lem:gcd_lem},  we must have ${\gcd(a,g) = \gcd(\gcd(a,f), g) \in \overline{D_1}}.$ 
Again, by \cref{lem:gcd_lem}, we have  
\[ \gcd(\gcd(a-gt, f),g) = \gcd(a-gt, g) =\gcd(a, g) \in \overline{D_1} .\]
Because $D_1 = \Phi_{f,g}^{-1}(\overline{D_1}) \cap \Div(f)$, this shows that $\gcd(a-gt, f) \in D_1$ and hence ${\gcd(a-gt, f) \in D}  $ as required.

Conversely, we claim that $(1)$ implies $(2)$. Suppose that $I$ is a homogeneous set in $G_f(D).$ We need to show that $D_1 = \Phi_{f,g}^{-1}(\overline{D_1}) \cap \Div(f)$. By \cref{prop:projection_morphism}, we always have $D_1 \subset \Phi_{f,g}^{-1}(D_1) \cap \Div(f)$. Therefore, it is sufficient to show that $\Phi_{f,g}^{-1}(\overline{D_1}) \cap \Div(f) \subset D_1.$ Let $h \in \Phi_{f,g}^{-1}(\overline{D_1}) \cap \Div(f)$. By definition, there exists $f_i \in D_1$ such that 
\[ \gcd(h,g) = \gcd(f_i, g).\]
By \cref{lem:elementary}, we can find $t \in \F_q[x]$ such that 
\[ \gcd(h-gt, f) = f_i.\]

This shows that $(h, gt) \in E(G_f(D)).$ Since $I$ is homogenous and $h \not \in I$, we conclude that $(h,0) \in E(G_f(D))$ as well. By definition,  $\gcd(h,f) \in D$. Since $h \mid f$, we conclude that $h \in D$ and hence $h \in D_1.$
\end{proof}

In general, it seems unclear how to check the conditions mentioned in \cref{thm:prime_criterior} explicitly. We discuss here a particular case when this can be done.

\begin{prop} \label{prop:sufficient_condition_prime}
   Let $f,g, I$ be as in \cref{thm:prime_criterior}. Assume further that $f_i \nmid f_j$ for all $i \neq j$ and $\gcd(f_1, f_2, \ldots, f_k)=1$. Then $I$ is a homogeneous set in $G_f(D)$ if and only if the following conditions hold.
   
   \begin{enumerate}
       \item For each $1 \leq i \leq k $,  $f_i \mid g$ for all $i$.
       \item    Furthermore, 
       \[ \rad \left(\frac{g}{f_i} \right) = \rad \left(\frac{f}{f_i} \right). \]
    \end{enumerate}
   In particular, if $f$ is squarefree then $f=g.$
   
\end{prop}

\begin{proof}
Let us first assume that $I$ is homogenous.   Let $g_i = \gcd(f_i,g).$ We claim that if $f_i \in D_1$ then $g_i =f_i$ and hence $f_i \mid g.$ In fact, we have $\gcd(g_i, g) = g_i = \gcd(f_i,g).$ This shows that $g_i \in \Phi_{f,g}^{-1}(\overline{D_1}) \cap \Div(f) = D_1$ (by \cref{thm:prime_criterior}). Since $f_i \nmid f_j$ for all $i \neq j$, we must have $f_i = g_i$ and hence $f_i \mid g.$ By our assumption, $\gcd(f_1, f_2, \ldots, f_k)=1$, and hence we must have $D_1 \neq \emptyset$. If $D_2 \neq \emptyset$ then for each $f_1 \in D_1$ and $f_2 \in D_2$, we have $f_1 \mid g \mid f_2$ which is a contradiction. Therefore, we must have $D_2 = \emptyset.$ In summary, we just prove that $f_i \mid g$ for all $1 \leq i \leq k.$ We now show that 

   \[ \rad \left(\frac{g}{f_i} \right) = \rad \left(\frac{f}{f_i} \right). \]
Suppose this is not the case. We can find a non-constant irreducible polynomial $h$ of such that $\gcd(h,\frac{g}{f_i})=1$ and $h\mid \frac{f}{f_i}.$ We then see that $\gcd(hf_i, g)=f_i =\gcd(f_i,g).$ This implies that $hf_i \in D$. Since $f_i \mid hf_i$ and $f_i \neq hf_i$, this leads to a contradiction.

Let us now show the converse. By \cref{thm:prime_criterior}, we need to show that if ${h \in \Phi_{f,g}^{-1}(\overline{D_1}) \cap \Div(f)}$ then $h \in D_1.$ Since $h \in \Phi_{f,g}^{-1}(\overline{D_1})$ we can find $f_i \in D_1$ such that $\gcd(h,g) = \gcd(f_i,g)=f_i.$ Let us write $h=f_i h_1$ with $\gcd(h_1, \frac{g}{f_i})=1.$ By our assumption that $\rad \left(\frac{g}{f_i} \right) = \rad \left(\frac{f}{f_i} \right)$, we must have $\gcd(h_1, \frac{f}{f_i})=1$ as well. Since $h_1 \mid \frac{f}{f_i}$, we must have $h_1=1$ and hence $h=f_i.$
\end{proof}
\begin{cor} 
    Let that $f$ be a squarefree polynomial. Suppose that $G_{f}(\{1\})$ is connected. Let $D = \{f_1, f_2, \ldots, f_k\}$ be a subset of $\Div(f)$ such that the following conditions hold. 
    \begin{enumerate}
        \item $f_i \nmid f_j$ for all $i \neq j.$
        \item $G_{f}(D)$ is connected and anti-connected.
    \end{enumerate}
    
     Then $G_{f}(D)$ is prime.
\end{cor}

\begin{proof}
   Suppose that $G_f(D)$ is not prime. By \cref{prop:prime_implies_ideal}, there exists a proper divisor $g \mid f$ such that the ideal $I$ generated by $g$ is homogeneous. By \cref{prop:sufficient_condition_prime}, we must have $g=f$ which is a contradiction. 
\end{proof}

\begin{rem}
    The proof of \cref{prop:sufficient_condition_prime} relies crucially on the divisibility relationship between $f_i$ and $f_j.$ It seems important to study this relationship systematically. {Since the completion of this paper, we have made some partial progress on this problem (see \cite{divisor_graphs}).}
\end{rem}

\begin{rem}
While we primarily focus on gcd-graphs over $\F_q[x]$ in this section, most statements have a straightforward analog for gcd-graphs over $\Z.$
    
\end{rem}
\section{Spectrum of gcd-graphs} \label{sec:spec}
The spectrum of gcd-graphs over $\Z$ is described by the theory of Ramanujan sums (see \cite[Section 4]{klotz2007some}), which in turn is a special case of Gauss sums (see \cite{lamprecht1953allgemeine}). As explained in \cite{chidambaram2023fekete}, these sums are precisely values of Fekete polynomials at certain $n$-roots of unity. One might wonder whether such a similar statement for the spectrum of gcd-graphs holds in the context of function fields. While we are not able to find an analog of Fekete polynomials over function fields, the theory of Ramanujan sums does have an interesting analogy as we will explain in this section. This, however, is sufficient for us to describe explicitly the spectrum of gcd-graphs over $\F_q[x].$

\subsection{Symmetric algebras}
A key point in the theory of $\Z/n\Z$-circulant graph is the fact that the character group of $\Z/n\Z$ is isomorphic to $\Z/n\Z$: 
\[ \Z/n\Z \cong \Hom(\Z/n\Z, \C^{\times}). \]
This isomorphism can be obtained as follows. Fix a primitive $n$-root of unity $\zeta_n$ in $\C$. Let $\chi_1\colon \Z/n\Z \to \C^{\times}$ be the character defined by $\chi_1(m) = \zeta_n^m$ for all $m \in \Z/n\Z.$ For each $a \in \Z/n\Z$, let $\chi_a = \chi_1^a$ be the character of $\Z/n\Z$ defined by $\chi_a(b) = \zeta_{n}^{ab} = \chi_1^a(b).$ The following proposition is standard. 
\begin{prop} \label{prop:character_z_n}
    The map $a \mapsto \chi_1^a$ gives an isomorphism between $\Z/n\Z$ and $\Hom(\Z/n\Z, \C^{\times}).$
\end{prop}
In summary, once we fix a primitive $n$-root of unity, the isomorphism ${\Z/n\Z \cong \Hom(\Z/n\Z, \C^{\times})}$ is obtained via the multiplicative structure on $\Z/n\Z$. We will use a similar approach in the function field case. We first recall the following definition. 

\begin{definition} (See \cite[Page 66-67]{lam2012lectures})
    Let $A$ be a finite dimensional commutative $\F_q$-algebra. $A$ is said to be a {\emph{symmetric}} $\F_q$-algebra if there exists an $\F_q$-linear functional $\lambda\colon A \to \F_q$ such that  the kernel of $\lambda$ contains no nonzero ideal of $A$. We call $\lambda$ a {\emph{non-degenerate linear functional}} on $A.$
\end{definition}

\begin{expl}
    Let $\F_{q^r}$ be a finite extension of $\F_q.$ Then, $\F_{q^r}$ equipped with the canonical  trace map $\Tr\colon \F_{q^r} \to \F_q$ is a symmetric $\F_q$-algebra. 
\end{expl}

The following lemma is rather standard.
\begin{lem} \label{lem:symmetric_F_q}
    Suppose that $A$ is a symmetric $\F_{q^r}$-algebra with a $\F_{q^r}$-linear functional $\lambda: A \to \F_{q^r}.$ Then $A$ is a symmetric $\F_q$-algebra where the linear functional is the composition of $\lambda$ and $\Tr\colon \F_{q^r} \to \F_q.$
\end{lem}

\begin{prop} \label{prop:dual_A}
Let $A$ be a symmetric finite dimensional $\F_q$-algebra and $\lambda\colon A \to \F_q$ an associated non-degenerate $\F_q$-linear functional. For each $a \in A$, let $\lambda_a\colon A \to \F_q$ be the $\F_q$-linear map defined by $\lambda_a(b) = \lambda(ab).$  Let $\Phi$ be the map $A \to \Hom_{\F_q}(A, \F_q)$ sending $a \mapsto \lambda_a.$ Then $\Phi$ is an isomorphism. 
\end{prop}

\begin{proof}
    Since $\lambda$ is non-degenerate, $\Phi$ is injective. Furthermore, because $A$ is finite dimensional over $\F_q$, $\dim(A) = \dim(\Hom_{\F_q}(A, \F_q)).$ Since $\Phi$ is $\F_q$-linear, it must be an isomorphism. 
\end{proof}

Fix a primitive $p$-root of unity $\zeta_p \in \C^{\times}$. Then, $\F_p$ is (non)-canonically isomorphic to the subgroup of $\C^{\times}$ generated by $\zeta_p$. If $A$ is an $\F_q$-algebra, then $(A,+)$ is a direct sum of several copies of $\F_p$. Therefore 
\[ \Hom((A,+), \C^{\times}) \cong \Hom_{\F_p}(A, \F_p).\]

By \cref{prop:dual_A} we have the following corollary, which is a direct analog of \cref{prop:character_z_n}.

\begin{cor} \label{cor:all_char_A}
    Let $A$ be a symmetric $\F_p$-algebra together with a non-degenerate functional $\lambda\colon A \to \F_p.$ For each $a \in A$, let $\lambda_a\colon A \to \C^{\times}$ defined by $\lambda_a(b) = \zeta_p^{\lambda(ab)}.$ Then $\lambda_a \in \Hom((A,+), \C^{\times}).$ Furthermore, the map $a \mapsto \lambda_a$ gives an isomorphism between $A$ and $\Hom((A,+), \C^{\times})$.
\end{cor}

We will now focus on the case $A=\F_q[x]/(f)$. We will show that it is a symmetric $\F_q$-algebra (and hence a symmetric $\F_p$-algebra as explained from \cref{lem:symmetric_F_q}). We will show this by constructing an explicit $\F_q$-linear functional on $A.$ We learned about this construction from \cite{kowalski2018exponential}.  Assume $\deg f=n$. Every element $g$ in $\F_q[x]/(f)$ can be written uniquely in the form 
\[ g = a_0(g)+a_1(g)x+\cdots+ a_{n-1}(g)x^{n-1} .\] 

We define $\psi\colon \F_q[x]/(f) \to \F_q$ by 
\[ \psi(g) = a_{n-1}(g).\]

\begin{prop}
    Suppose $g \in \F_q[x]/(f)$ such that $\psi(hg)=0$ for all $h \in \F_q[x]/(f).$ Then $g=0$ in $\F_q[x]/(f).$
\end{prop}

\begin{proof}
    Let us write $g = a_0(g)+a_1(g)x+\cdots+ a_{n-1}(g)x^{n-1}$. We will prove by induction that $a_{n-k}(g)=0$ for $1 \leq k \leq n.$
In fact, since $\psi(g)=0$, we know that $a_{n-1}(g)=0.$ Consequently, the statement is true for $k=1.$ Let us assume that it has been shown for all $1 \leq k \leq m <n$. We claim that it is also true for $m+1$, namely $a_{n-m-1}(g)=0$ as well. In fact, we have 
\begin{equation*}
\begin{split} 
x^{m}g &= x^{m}a_0(g)+\cdots+ a_{n-m-1}(g)x^{n-1} + a_{n-m}(g)x^{n} + \cdots+a_{n-1}(g)x^{m+n-1} \\
 &= x^{m}a_0(g)+\cdots+ a_{n-m-1}(g)x^{n-1}.
\end{split}
\end{equation*}
Here, the last equation follows from the fact that $a_{n-m}=a_{n-m+1}= \ldots = a_{n-1}=0$ by the induction hypothesis. Consequently 
\[ a_{n-m-1}(g) = \psi(x^{m} g) =0.\]
By the induction principle, we conclude that $a_{n-k}=0$ for all $1 \leq k \leq n.$
    \end{proof}

\begin{cor}
    $\psi$ is a non-degenerate $\F_q$-linear functional on $\F_q[x]/(f).$ Consequently, $\F_q[x]/(f)$ is a symmetric $\F_q$-algebra. 
\end{cor}

By \cref{lem:symmetric_F_q}, under the composition $\F_q[x]/(f) \xrightarrow{\psi} \F_q \xrightarrow{\Tr} \F_p$ where $\text{Tr}$ is the trace map,  $\F_q[x]/(f)$ becomes a symmetric $\F_p$-algebra.  By \cref{cor:all_char_A}, we have the following proposition.

\begin{prop} \label{prop:bijection_characters}

There exists a bijection 
\[ \F_q[x]/(f) \longleftrightarrow  \Hom((\F_q[x]/(f),+), \C^{\times}), \quad a \longleftrightarrow \{\psi_a\},\]
where 
\[ \psi_a\colon \F_q[x]/(f) \to \C^{\times}\] is given by 
\[ \psi_a(b) = \zeta_p^{\Tr(\psi(ab))}, \forall b \in \F_q[x]/(f).\]
\end{prop}

We remark that over $\Z$, if $\zeta_n$ is a primitive $n$th-root of unity, then for each divisor $m \mid n$, $\zeta_n^{\frac{n}{m}}$ is a primitive $m$-root of unity.  An analogous statement holds for $\F_q[x]$ as well. 

\begin{prop} \label{prop:induced_primitive}

Let $f \in \F_q[x]$ and $\psi\colon \F_q[x]/(f) \to \F_q$ be a non-degenerate linear functional. Let $g$ be a divisor of $f$ and $\psi_g\colon \F_q[x]/g \to \F_q$ be the function defined by 
\[ \psi_g(a)  = \psi \left(\frac{f}{g} a \right) .\]
Then $\psi_g$ is a non-degenerate linear functional on $\F_q[x]/g.$
\end{prop}

\begin{proof}
    It is clear from the definition that $\psi_g$ is $\F_q$-linear. We only need to show that it is non-degenerate. In fact, suppose to the contrary that the kernel of $\psi_g$ contains a non-zero ideal $I$ in $\F_q[x]/g.$ Since $\F_q[x]$ is a PID, $I$ must be of the form $I = \langle h \rangle$ for some $h\mid g.$ We then see that $\langle h \frac{f}{g} \rangle$ belongs to the kernel of $\psi.$ Because $\psi$ is non-degenerate, this implies that $h \frac{f}{g} = 0$ in $\F_q[x]/(f).$ In other words, $g\mid h$ or equivalently $h=0$ in $\F_q[x]/g.$ This shows that $I=0$, which is a contradiction. 
\end{proof}

\subsection{Ramanujan sums over $\F_q[x]$}

We now introduce the definition of Ramanujan sum over $\F_q[x].$

\begin{definition}(Ramanujan sums over $\F_q[x]$) \label{def:ramanujan_sum}
    Let $f,g \in \F_q[x]$ be polynomials {such that $f$ is monic}. Let $\psi\colon \F_q[x]/(f) \to \F_q$ be a non-degenerate linear functional on $\F_q[x]/(f)$. The {\emph{Ramanujan sum}} $c(g,f)$ is defined as 
    \[ c(g,f) = c_{\psi}(g,f) = \sum_{a \in (\F_q[x]/(f))^{\times}} \zeta_p^{\Tr(\psi(ga))}. \]
\end{definition}

\begin{rem}
We remark also that at first glance, $c(g,f)$ depends on the choice of $\psi.$ However, as we explain in what follows, it does not depend on the choice of $\psi$ as long as we make sure that $\psi$ is non-degenerate. This is similar to the case over $\Z$: Ramanujan sums do not depend on the choice of a primitive $n$th-root of unity. In fact, we will show that there is an explicit formula for $c(g,f)$ similar to the case of Ramanujan sum over $\Z.$
\end{rem}

\begin{rem}
    Ramanujan sums are a special case of Gauss sums as defined and studied in \cite[Definition 1]{lamprecht1953allgemeine}. In fact, they are Gauss sums for the principal Dirichlet characters on $\F_q[x]/(f).$ It would be rather interesting if we could define Fekete polynomials for these principal Dirichlet characters (the case over $\Z$ was studied in \cite{chidambaram2023fekete}). 
\end{rem}

We first recall the following standard lemma in group theory. 
\begin{lem} \label{lem:sum_of_chars}

Let $G$ be a finite group and let $\chi\colon G \to \C^{\times}$ be a non-trivial character. Then 
\[ \sum_{g \in G} \chi(g)= 0.\]
\end{lem}

\begin{proof}
    Since $\chi$ is non-trivial, there exists $h \in G$ such that $\chi(h) \neq 1.$ We then have 
    \[ \sum_{g \in G} \chi(g) = \sum_{h \in G} \chi(hg) = \chi(h) \sum_{g \in G} \chi(g).\]
    Because $\chi(h) \neq 1$, we must have $\sum_{g \in G} \chi(g)=0.$
\end{proof}

{The  {M{\"o}bius function} $\mu$ on $\F_q[x]$ is defined  in the same way as in the integer case, which we recall as follows: for a non-zero polynomial $f\in \F_q[x]$,
\[
\mu(f)=\begin{cases}
    (-1)^r &\text{if $f=a P_1\cdots P_r$,   }\\
    0 & \text{ if $f$ is not square-free.}
\end{cases}
\]
Here $a$ is a non-zero constant and $P_1,\ldots,P_r$ are distinct monic irreducible polynomials dividing $f$ ($r$ could be $0$). 
}

\begin{prop} \label{prop:ramanujan_sum_1}
Let $\psi$ be a non-degenerate linear functional on $\F_q[x]/(f)$. For each $f \in \F_q[x]$
\[ c(1,f) = c_{\psi}(1,f)= \mu(f), \]
where $\mu$ is the M{\"o}bius function on $\F_q[x]$. 
\end{prop}

\begin{proof}
    If $f = \prod_{i=1}^d f_i^{\alpha_i}$ where $\alpha_i \in \N$ and $f_i$ are irreducible, then by the Chinese remainder theorem and \cite[Satz 1]{lamprecht1953allgemeine} we have 
    \[ c_{\psi}(1,f) = \prod_{i=1}^d c_{\psi_i}(1,f_i^{\alpha_i}),\]
    where $\psi_i\colon \F_q[x]/(f)_i^{\alpha_i} \to \F_q$ is a non-degenerate linear functional induced by $\psi.$
    
    Therefore, it is enough to consider the case $f = f_1^{a_1}$ where $f_1$ is irreducible. We have 
    \[\begin{aligned}
    c_{\psi}(1,f) &= \sum_{a \in (\F_q[x]/f)^{\times}} \zeta_p^{\Tr(\psi(a))}  
       = \sum_{a \in \F_q[x]/f} \zeta_p^{\Tr(\psi(a))} - \sum_{a \in \F_q[x]/f_1^{a_1-1}} \zeta_p^{\Tr(\psi(f_1 a))}.      
    \end{aligned} 
    \]
By \cref{lem:sum_of_chars}
\[ \sum_{a \in \F_q[x]/f} \zeta_p^{\Tr(\psi(a))} = 0.\]
Similarly, if $a_1 \geq 2$ then by \cref{lem:sum_of_chars} and \cref{prop:induced_primitive}
\[ \sum_{a \in \F_q[x]/(f_1^{a_1-1})} \zeta_p^{\Tr(\psi(f_1 a))}  = 0 .\] 
On the other hand if $a_1=1$ then $\sum_{a \in \F_q[x]/(f_1^{a_1-1})} \zeta_p^{\Tr(\psi(f_1 a))} =1$. We conclude that 
\[ c(1,f) = c_{\psi}(1,f) = c_{\psi}(1, f_{1}^{a_1}) = \mu(f_1^{a_1}) = \mu(f).
\qedhere\]
\end{proof}

For general $g$, we have the following theorem.

\begin{thm}
    Let $f,g \in \F_q[x]$ be polynomials such that $f$ is monic. Then (compare with \cite[Equation 9]{klotz2007some})
    \[ c(g,f)  =  \mu(t) \dfrac{\varphi(f)}{\varphi(t)}, \quad \text{where} \quad t = \dfrac{f}{\gcd(f,g)}. \] 
    Here $\varphi$ is the Euler totient function on $\F_q[x]$; i.e, $\varphi(h) = |(\F_q[x]/(h))^{\times}|.$ {In particular, $c(g,f)$ is an integer.}
\end{thm}
\begin{proof} Let $\psi\colon \F_q[x]/(f) \to \F_q$ be a non-degenerate linear functional and 
\[c(g,f)= c_{\psi}(g,f).\]
Let $h=\gcd(g,f)$ and $t=\dfrac{f}{h}$. The canonical projection map 
\[ {(\F_q[x]/(f))^{\times} \to (\F_q[x]/(t))^{\times}} ,\] 
has kernel of size $\frac{\varphi(f)}{\varphi(t)}.$ Furthermore, for each $a \in (\F_q[x]/(t))^{\times}$, if $a_1, a_2$ are two preimages of $a$ in $(\F_q[x]/(f))^{\times}$  then $a_2=a_1+tu$, for some $u\in \F_q[x]/(f)$. In this case, one has 
\[\psi(ga_2)= \psi(ag_1+gtu)=\psi(ga_1)+\psi(\frac{gu}{h}f)=\psi(ga_1).\] 
Consequently, we have 
\[
\begin{aligned}
c_{\psi}(g,f) &= \sum_{a \in (\F_q[x]/(f))^{\times}} \zeta_p^{\Tr(\psi(ga))} = \frac{\varphi(f)}{\varphi(t)}\sum_{a \in (\F_q[x]/(t))^{\times}} \zeta_p^{\Tr(\psi(ga))}
\\ 
&=  \frac{\varphi(f)}{\varphi(t)}\sum_{a \in (\F_q[x]/(t))^{\times}} \zeta_p^{\Tr(\psi(\frac{f}{t}\frac{g}{h}a))}= \frac{\varphi(f)}{\varphi(t)}\sum_{a \in (\F_q[x]/(t))^{\times}} \zeta_p^{\Tr(\psi_t(\frac{g}{h}a))}.
\end{aligned}\] 
Here $\psi_t\colon \F_q[x]/(t) \to \F_q$ is the linear functional defined by $\psi_t(b) = \psi(\frac{f}{t}b)=\psi(hb)$. Note that $\psi_t$ is non-degenerate.

Since $g/h \in (\F_q[x]/(t))^{\times}$, we see that 
\[ \sum_{a \in (\F_q[x]/(t))^{\times}} \zeta_p^{\Tr(\psi_t(\frac{g}{h}a))} =\sum_{b \in (\F_q[x]/(t))^{\times}} \zeta_p^{\Tr(\psi_t(t))} = c_{\psi_t}(1,t) = \mu(t).\]
Here the last equality follows from   \cref{prop:ramanujan_sum_1}.
We conclude that 
\[ c(g,f)= c_{\psi}(g,f)  = \frac{\varphi(f)}{\varphi(t)} \mu(t). \qedhere\]
\end{proof}

\subsection{Spectrum of $G_{f}(D)$}
Let $G$ be an abelian group and $S$ a symmetric subset of $G$ such that $0 \not \in S$. The adjacency matrix of the Cayley graph $\Gamma(G,S)$ is a $G$-circulant matrix. By the $G$-circulant theorem, the spectrum of $\Gamma(G,S)$ is the collection of the sums  $\sum_{s \in S} \chi(s)$, where $\chi$ runs over the set of all characters of $G$ (see \cite[Section 1.2, {Equation (1.1')}]{kanemitsu2013matrices}). When $G=\F_q[x]/(f)$, we know by \cref{prop:bijection_characters} that the character of $G$ is parameterized by $G$ itself. Consequently, we have the following proposition. 

\begin{prop} \label{prop:spectrum_cayley}
    Let $S$ be a subset of $\F_q[x].$ Then the {eigenvalues} of the Cayley graph $\Gamma(\F_q[x]/(f), S)$ are given by the set 
    \[ \left \{  \sum_{s \in S} \zeta_p^{\Tr(\psi(gs))}\right\}_{g \in \F_q[x]/(f)}.\]
\end{prop}

Let us now focus on the case of gcd-graphs; i.e., $S= S_D$. We have the following lemma. 

\begin{lem} \label{lem:spectrum_for_single_f_i}
Let $h \mid f$ be a monic divisor of $f.$ Then for each $g \in \F_q[x]$
\[ \sum_{a \in \F_q[x]/(f), \gcd(a,f)=h} \zeta_{p}^{\Tr(\psi(ag))} = c \left(g, \frac{f}{h} \right).\]
    
\end{lem}

\begin{proof}
    For $a \in \F_q[x]$, $\gcd(a,f)=h$ if and only if $a = hb$ where $\gcd(b, \frac{f}{h})=1.$ Therefore, the above sum can be rewritten as 
    \[ \sum_{\substack{b \in \F_q[x]/(f/h)\\ \gcd(a,f/h)=1}} \zeta_{p}^{\Tr(\psi(bgh))} = \sum_{\substack{b \in \F_q[x]/(f/h)\\ \gcd(b,f/h)=1}} \zeta_{p}^{\Tr(\psi_h(bg))}.\]
    Here $\psi_h\colon \F_q[x]/(f/h) \to \F_q$ is the functional given by $\psi_h(x) = \psi(hx).$ By \cref{prop:induced_primitive}, $\psi_h$ is a non-generate linear functional on $\F_q[x]/(f/g).$ Therefore, we conclude that 
    \[\sum_{\substack{a \in \F_q[x]/(f)\\ \gcd(a,f)=h}} \zeta_{p}^{\Tr(\psi(ag))}= \sum_{\substack{b \in \F_q[x]/(f/h)\\ \gcd(b,f/h)=1}} \zeta_{p}^{\Tr(\psi_h(bg))} = c \left(g, \frac{f}{h} \right). \] 
\end{proof}
By \cref{prop:spectrum_cayley} and \cref{lem:spectrum_for_single_f_i}, we have the following theorem. 
\begin{thm} \label{thm:integer_eigens}
Let $f$ be a monic polynomial and $D= \{f_1, f_2, \ldots, f_k \}$ where $f_i \mid f.$ Then, the {eigenvalues} of $G_{f}(D)$ are given by the set 
    \[ \left \{ \sum_{i=1}^k c \left(g, \frac{f}{f_i} \right)\right \}_{g \in \F_q[x]/(f)}.
 \] 
  {In particular, all eigenvalues of the gcd-graph $G_f(D)$ are integers. }  
    \end{thm}
\begin{expl}

    The characteristic polynomial of the graph described in \cref{fig:z_25} has the following factorization 
 \[ (x - 4) (x - 1)^4  (x + 2)^4. \] 
 Similarly,  the characteristic polynomial of the graph described in \cref{fig:z_26} has the following factorization 
\[ (x - 3)(x + 3) (x - 1)^3 (x + 1)^3 .\] 
In particular, the eigenvalues of these graphs are integers as shown in \cref{thm:integer_eigens}. {We remark that these spectral calculations have also been verified using Ramanujan sums, as explained in \cref{thm:integer_eigens}. These verifications are described in \cite{Nguyen_gcd_graph}, which contains the code to calculate various arithmetic functions over $\F_q[x]$. In particular, it has a function to calculate the general Ramanujan sum $c(g,f)$ for all $g, f \in \F_q[x].$  }

{
For example, let us consider the graph $G_f(D)$ described in \cref{fig:z_25}, where ${f=x(x+1)\in \F_3[x]}$ and $D=\{x,x+1\}$. We shall identify $\F_3[x]/(f)$ with
\[\{0,1,2,x, x+1,x+2,2x,2x+1,2x+2\}.\] 
For $g=0$, we have $c(0,x)=\mu(1)\dfrac{\varphi(x)}{\varphi(1)}=2$, and similarly $c(0,x+1)=2$. Hence the eigenvalue of $G_{f}(D)$ associated with $g=0$ is $c(0,x)+c(0,x+1)=4$. (The eigenvalue $4$ is the degree of $G_{f}(D)$, and it appears with multiplicity $1$.)}

{
For $g$ coprime to $x$, we have $c(g,x)=\mu(x)\dfrac{\varphi(x)}{\varphi(x)}=-1$. Similarly, for $g$ coprime to $x+1$, we have $c(g,x+1)=-1$. 
We also have  
\[ c(x,x)=c(2x,x)=\mu(1)\dfrac{\varphi(x)}{\varphi(1)}=2, \]
and  
\[c(x+1,x+1)=c(2x+2,x+1)=2. \]}

{From the above discussion, we see that if $g\in \{1,2,x+2,2x+1\}$, then the eigenvalue of $G_{f}(D)$ associated with $g$ is 
\[c(g,x)+c(g,x+1)=(-1)+(-1)=-2. \]
(The eigenvalue $-2$ appears with multiplicity 4.) Finally, if $g\in \{x,x+1,2x,2x+2\}$, then the eigenvalue of $G_{f}(D)$ associated with $g$ is 
\[c(g,x)+c(g,x+1)=2+(-1)=1. \]
(The eigenvalue $1$ appears with multiplicity 4.) Our calculations show that the spectrum of $G_{f}(D)$, as described in \cref{thm:integer_eigens}, coincides with the spectrum obtained from the adjacency matrix.}
\end{expl}

As we discussed in the introduction, it is known that a $\Z/n\Z$-circulant graph has an integral spectrum if and only if it is a gcd-graph (see \cite[Theorem 7.1]{so2006integral} ). One may ask whether the same statement is true for graphs associated with $\F_q[x]/(f).$ The answer is no in general. In fact, we have the following general observation. 

\begin{prop} \label{prop:integral}
    Let $S$ be a symmetric subset of $\F_q[x]/(f)$ such that $0 \not \in S.$ Suppose further that $\F_{p}^{\times}S = S.$ Then, the Cayley graph $\Gamma(\F_q[x]/(f), S)$ has an integral spectrum. 
\end{prop}

\begin{proof}
    We define the following equivalence relation on $\F_q[x]/(f)$. For $u, v \in \F_q[x]/(f)$, we say that $u \sim v$ if $u = kv$ where $k \in \F_p^{\times}.$ By \cref{prop:spectrum_cayley}, it is enough to show that for each $g \in \F_q[x]/(f)$,  $\sum_{s \in S} \zeta_p^{\Tr(\psi(gs))} \in \Z$.  Because $\F_p^{\times} S = S$, this sum can be written as 
\[ \sum_{[s] \in S/\sim } \left( \sum_{k \in \F_{p}^{\times}} \zeta_p^{\Tr(\psi(gks))} \right) = \sum_{[s] \in S/\sim } \left( \sum_{k \in \F_{p}^{\times}} (\zeta_p^{\Tr(\psi(gs))})^k \right) . \]
    We know that 
\[
\sum_{k \in \F_{p}^{\times}} (\zeta_p^{\Tr(\psi(gs))})^k = 
\begin{cases} 
-1 & \text{if } \Tr(\psi(gs)) \neq 0 \\
p-1 & \text{if } \Tr(\psi(gs))=0. 
\end{cases}
\]
This shows that $\sum_{s \in S} \zeta_p^{\Tr(\psi(gs))} \in \Z$ for each $g \in \F_q[x]/(f).$ 
\end{proof}

\begin{rem}
    Let $\F_q$ be a finite field such that $\F_q \neq \F_p.$ Let $f=x$. In this case $\F_q[x]/(f) \cong \F_q.$ Let $V$ be proper $\F_p$-subspace of $\F_q$ and $S = V \setminus \{0 \}.$ Then by \cref{prop:integral}, $\Gamma(\F_q[x]/(f), S)$ has an integral spectrum even though it is not a gcd-graph. {In a recent work, we show that the converse of this statement is also true; namely, $\Gamma(\F_q[x]/(f), S)$ is integral if and only if $S \cup \{0\}$ is a vector space over $\F_p$ (see \cite[Theorem 1.1]{nguyen2024integral}.)}

\end{rem}


\section{Perfect gcd-graphs} \label{sec:perfect}
A graph $G$ is said to be {\emph{perfect}} if, for every induced subgraph $H$ of $G$, the chromatic number of $H$ equals the size of its maximum clique. Perfect graphs play a fundamental role in the study of graph coloring and cliques. They encompass several important families of graphs and provide a unified framework for results relating to colorings and cliques within these families. Moreover, many central problems in combinatorics can be rephrased as questions about whether certain associated graphs are perfect (see \cite{chvatal1984topics, dilworth1987decomposition, mirsky1971dual}). For these reasons, it seems interesting to study whether $G_{f}(D)$ is perfect.

In \cite[Theorem 9.5]{unitary}, the authors classify all perfect unitary Cayley graphs associated with finite commutative rings. Specifically, for a ring $R=R_1 \times R_2 \times \cdots \times R_t$ where $R_i$ are finite local rings such that $|R_1| \leq |R_2| \leq \cdots \leq |R_t|$, the unitary Cayley graph on $R$ is perfect if and only if one of the following conditions hold: 

\begin{enumerate}
    \item The residue field of $R_1$ is $2$. In this case, $G_{R}$ is bipartite and hence perfect; 
    \item $R$ is either a local ring or a product of two local rings; i.e, $t \leq 2.$
\end{enumerate}

A direct consequence of \cite[Theorem 9.5]{unitary} for the unitary Cayley graph over $\Z$ is the following. 

\begin{cor}
    The unitary Cayley graph on $\Z/n\Z$ is perfect if and only if one of the following conditions holds. 

    \begin{enumerate}
        \item $2 \mid N$.
        \item $\omega(N) \leq 2$ where $\omega(N)$ is the number of distinct irreducible factors of $N.$
    \end{enumerate}
\end{cor}

Over the polynomial ring $\F_q[x]$, we have an analogous statement. 

\begin{cor} \label{cor:unitary_perfect}
    The unitary Cayley graph $G_{f}(\{1 \})$ on $\F_q[x]/(f)$ is perfect if and only if one of the following conditions holds. 

    \begin{enumerate}
        \item $\F_q= \F_2$ and $gcd(f, x(x+1)) \neq 1.$ 
        \item $\omega(f) \leq 2$ where $\omega(f)$ is the number of distinct irreducible factors of $f.$
    \end{enumerate}
\end{cor}
We will use \cref{cor:unitary_perfect} together with patterns of $G_{f}(D)$ discovered through our experiments in Sagemath and the Python package Networkx in order to find some sufficient conditions for $G_{f}(D)$ to be \textit{non-perfect}. In particular, we will exploit the fact that certain induced subgraphs of $G_{f}(D)$ are naturally isomorphic to the unitary Cayley graphs on some quotient rings of $R = \F_q[x]/(f).$  More precisely, by \cref{lem:induced_graph_ideal}, we have the following observation.


\begin{prop} \label{prop:induced_perfect_p_count}
    Suppose $g \in D$ such that $g \nmid f_i$ for all $i$ such that $f_i \neq g.$ Then the induced graph on $I_g$ is naturally isomorphic to the unitary Cayley graph $G_{f/g}(\{1 \}).$ Furthermore, if $\omega(f/g) \geq 3$ and $\F_q \neq \F_2$, then $G_{f/g}(D)$ is not perfect and hence $G_{f}(D)$ is not perfect. 
\end{prop}

We discuss a case where we can apply \cref{prop:induced_perfect_p_count} rather directly. 

\begin{prop} \label{prop:perfect_k_3}
    Let $f \in \F_q[x]$ and ${D} = \{f_1, f_2, \ldots, f_k\}$ be a subset of divisors of $f$. Suppose that the following conditions hold. 

    \begin{enumerate}
    \item $\deg(f_i) \geq 1.$
        \item $f_i$'s are pairwisely relatively prime; i.e, $\gcd(f_i, f_j) =1$ for all $i \neq j.$
        \item $k \geq 3.$
        \item $\F_q \neq \F_2.$
    \end{enumerate}
    Then $G_{f}(D)$ is not a perfect graph. 
\end{prop}
\begin{proof}
    If $k \geq 4$ then 
    \[\omega(f/f_1)  \geq \omega(f_2)+\omega(f_3)+\omega(f_4) \geq 3 .\] 
    By applying \cref{prop:induced_perfect_p_count} for $g=f_1$, we conclude that $G_{f}(D)$ is not perfect. Let us assume now that $k=3$. By the same argument, it is enough to consider the case $\omega(f_1)=\omega(f_2) =\omega(f_3)=1.$ Furthermore, if $ f \neq f_1 f_2 f_3$, then there exists an index $i \in \{1, 2, 3\}$ such that $\omega(f/f_i) \geq 3.$ As a result, $G_{f}(D)$ is not perfect by \cref{prop:induced_perfect_p_count}.
    Let us now consider the case $f=f_1f_2f_3$. In this case 
    \[ \F_q[x]/(f) \cong \F_q[x]/(f_1 )\times \F_q[x]/(f_2) \times \F_q[x]/(f_3). \]
    Under this isomorphism, we can identify $V(G_{f}(D))$ as the set of all triples $(a_1, a_2, a_3)$ where $a_i \in \F_q[x]/(f_i)$ for $1 \leq i \leq 3.$ 
    Furthermore, two vertices $(a_1,a_2,a_3)$ and $(b_1, b_2, b_3)$ are adjacent if and only if there exists an index $i \in \{1, 2, 3\}$ such that $a_i=b_i$ and ${(a_j-b_i) \in (\F_q[x]/(f_j))^{\times}}$ if $j \neq i$. Using Sagemath, we can find the following induced $5$-cycle in $G_{f}(D)$ 
    \[ (0, 0, 0) \to (\alpha, \alpha, 0) \to (1, \alpha, 1) \to (1, 1, \alpha) \to (\alpha, 1, 0),\]
    where $\alpha \in \F_q \setminus \{0, 1 \}.$ This show that $G_{f}(D)$ is not perfect. 
\end{proof}
The case $k=2$ is a bit more challenging. After some experiments with Sagemath, we found the following statement. 

\begin{prop} \label{prop:perfect_k_2}
        Let $f \in \F_q[x]$ and ${D} = \{f_1, f_2 \}$ be a subset of divisors of $f$ such that $\gcd(f_1,f_2)=1$. Suppose that the following conditions hold. 
        \begin{enumerate}
        \item $\F_q \neq \F_2$.
        \item $\omega(f_1) \geq 2$, $\omega(f_2) \geq 1$. 
        \end{enumerate}
        Then, $G_{f}(D)$ is not perfect. 
\end{prop}

\begin{proof}
    Let us assume to the contrary that $G_f(D)$ is perfect. Since $\gcd(f_1,f_2)=1$, we have $f_1 f_2 \mid f$ and hence
    \[ \omega(f) \geq \omega(f_1)+\omega(f_2).\]
    If $\omega(f)> \omega(f_1)+\omega(f_2)$ then $\omega(f/f_2) > \omega(f_1) \geq 2.$ By \cref{prop:induced_perfect_p_count}, we know $G_f(D)$ is not perfect. Therefore, we must have $\omega(f)= \omega(f_1)+\omega(f_2).$

    Let $g= f_1 f_2.$  By \cref{prop:sufficient_condition_prime}, the ideal generated by $g$ is a homogeneous set in $G_f(D)$ and furthermore 
    $G_f(D)$ is isomorphic to the wreath product ${G_{g}(\{f_1, f_2\}) * G_{f/g}(\emptyset)}.$ Since $G_f(D)$ is perfect, $G_{g}(\{f_1, f_2\})$ is perfect as well. Because $\omega(f_1) \geq 2$, we can find $h_1, h_2 \in \F_q[x]$ such that $f_1 = h_1 h_2$ and $\gcd(h_1,h_2)=1.$ By the Chinese remainder theorem 
    \[ \F_q[x]/(g) \cong \F_q[x]/(h_1) \times \F_q[x]/(h_2) \times \F_q[x]/(f_2). \]

    Under this isomorphism, we can identify $\F_q[x]/(g)$ with the set of tuples $(a_1, a_2, a_3)$ such that $a_1 \in \F_q[x]/(h_1), a_2 \in \F_q[x]/(h_2), a_3 \in \F_q[x]/(f_2).$ Furthermore, two vertices $(a_1,a_2,a_3)$ and $(b_1, b_2, b_3)$ are adjacent if and only if one of the following conditions happens 
    \begin{enumerate}
        \item $a_1 = b_1, a_2 = b_2$ and $a_3 - b_3 \in (\F_q[x]/(f_2))^{\times}$,
        \item $a_3=b_3$ and $a_i-b_i \in (\F_q[x]/(h_i))^{\times}$ for $i \in \{1,2\}.$
    \end{enumerate}
    Using Sagemath, we can find the following $7$-cycle in $G_g(\{f_1, f_2 \})$
    \[ (0, 0, 0) \to (1,1,0) \to (\alpha, 0, 0) \to (\alpha, 0, 1) \to (0, 1,1) \to (\alpha, \alpha, 1) \to (0, 0, 1),\]
    where $\alpha \in \F_q \setminus \{0,1 \}.$ We conclude that $G_g(\{f_1, f_2\})$ is not a perfect graph, which is a contradiction. 
\end{proof}

\begin{rem}
    Curious readers might wonder why we choose a $7$-cycle in the proof of \cref{prop:perfect_k_2} instead of choosing a $5$-cycle as in the proof of \cref{prop:perfect_k_3}. The reason is that when $\omega(f_1)=2, \omega(f_2)=1$ and $f=f_1f_2$, our code cannot find a $5$-cycle in $G_f(D).$ It seems interesting to investigate whether this is always the case. 
\end{rem}

\begin{rem}
    The only remaining case that we miss is when $\omega(f_1)=\omega(f_2)=1.$ We distinguish the following cases. 

    \smallskip
\noindent (1) If $f=f_1f_2$ where both $f_1,f_2$ are irreducible then $G_f(D)$ is isomorphic to the complement of $G_{f}(\{1 \})$ which is perfect by \cref{cor:unitary_perfect}. Therefore, $G_f(D)$ is perfect as well. 

\smallskip

\noindent (2) On the other hand, if either $f_1$ or $f_2$ is reducible, our code can always find a $7$-cycle in $G_f(D).$ Unfortunately, we cannot find a universal pattern in this case. It would be quite interesting to solve this puzzle completely.

\end{rem}

\section{Induced subgraphs of gcd-graphs} \label{sec:induced}

A theorem of Erdős and Evans (see \cite{erdos1989representations}) says that every graph $G$ is an induced subgraph of the unitary Cayley graph on $\Z/n\Z$ for some squarefree $n.$ Using this result, it is shown in \cite{kiani2015unitary} that for a given finite graph $G$ and a finite field $F$, $G$ is an induced subgraph of the unitary graph of a matrix algebra $M_d(F)$ for some value of $d$ (they also provide some precise upper-bound on $d$ when $G$ is the complete graph $K_m$). In light of these results, it seems interesting to ask whether a graph $G$ can be realized as an induced subgraph of $G_{f}(D)$ for some choice of $f, D$ and $\F_q.$ Under some rather mild conditions, the answer is yes as we will show below. 
{For the related definitions we refer the reader to \cref{sec:background}.}
First, we introduce the following observation.

\begin{lem} \label{lem:add_complete_graph}
    If $G$ is an induced subgraph of $H$, then $G$ is also an induced subgraph of the {tensor} product $H \times K_{|G|}$ where $K_{|G|}$ is the complete graph on $|G|$-nodes.
\end{lem}
\begin{proof}
    Let $f\colon G \to H$ be a graph morphism that makes $G$ into an induced subgraph of $H.$ Let us index $V(G) =\{v_1, v_2, \ldots, v_{|G|} \}$. Let $\hat{f}\colon G \to H \times K_{|G|}$ be the map defined by 
    \[ \hat{f}(v_i) = (f(v_i), i) .\]
    We can see that $\hat{f}$ is a graph morphism which turns $G$ into an induced subgraph of $H \times K_{|G|}.$
\end{proof}

We are now ready to prove an analog of Erdős-Evans's theorem in the function fields case. 

\begin{prop} \label{prop:induced_unitary}
   Let $G$ be a fixed graph and $q$ a fixed prime power. There exists a positive integer $r$ such that for each $d \geq r$, {there exists} $F_d \in \F_q[x]$ satisfying the following conditions 
   \begin{enumerate}
       \item $\omega(F_d) = d$,
       \item $G$ is an induced subgraph of the unitary Cayley graph $G_{F_d}(\{1 \})$.
   \end{enumerate}

\end{prop}

\begin{proof}
By Erdős and Evans's theorem \cite{erdos1989representations}, there exists a squarefree integer $n \in \N$ such that $G$ is an induced subgraph of the unitary graph $G_n$ on $\Z/n\Z.$ Let $r = \omega(n)$; i.e,  $n=p_1p_2 \ldots p_r$ be the prime factorization of $n$. Then 
\[ G_n \cong \prod_{i=1}^r G_{p_i} \cong \prod_{i=1}^r K_{p_i}.\]
For each $i \in \N$ and $m_i \in \N$, by Galois theory for finite fields, there exists a polynomial $h_i$ of degree $m_i$ such that $h_i$ is irreducible over $\F_q[x].$ We then see that $G_{h_i}(\{1\})$ is isomorphic to the complete graph $K_{q^{m_i}}.$ Let us choose $m_i$ so that $q^{m_i}>n$ for all $i$. For each $d \geq r$, let $F_d = h_1 h_2 \ldots h_d.$ We then have 
\[ G_{F_d}(\{1\}) \cong \prod_{i=1}^d G_{h_i}(\{1\}) \cong \prod_{i=1}^d K_{q^{m_i}}.\]
Since $q^{m_i}>n$, we know that $G_n$ is an induced subgraph of $G_{F_r}(\{1\}).$ By \cref{lem:add_complete_graph} , $G$ is an induced subgraph of $G_{F_d}(\{1\})$ as well. 
\end{proof}

{We obtain the following result which says in particular that every graph is an induced subgraph of a gcd-graph.}
\begin{cor} \label{cor:induced_graph}
    Let $G$ be a fixed graph and $k$ a fixed positive integer. Then,  there exist a polynomial $f$ and a subset $D = \{f_1, f_2, \ldots, f_k\}$ of divisors of $f$ such that $G$ is an induced subgraph of $G_{f}(D).$
\end{cor}

\begin{proof}
    Let $h_0$ be an arbitrary polynomial in $\F_q[x].$ By \cref{prop:induced_unitary}, there exists a positive integer $d \geq k$ and a polynomial $h$ of the form $h = h_1 h_2  \ldots h_{d}$ such that $G$ is an induced subgraph of $G_{h}(\{1\}).$ Let us choose $f = h_0 h$ and   
    \[ \{f_1, f_2, \ldots, f_k \} = \{h_1, h_2, \ldots, h_k \}. \]
    Let $I$ be the ideal generated by $h_0$ in $\F_q[x]/(f).$ By \cref{lem:induced_graph_ideal} the induced graph on $I$ is naturally isomorphic to the unitary Cayley graph $G_{h}(\{1 \})$. This shows that $G$ is an induced subgraph of $G_{f}(D).$
\end{proof}

One may wonder whether the following stronger form of \cref{cor:induced_graph} holds. 
    Let $G$ be a fixed graph and $f_1, f_2, \ldots, f_k$ fixed polynomials. Does there exist a polynomial $f$ such that 
    \begin{enumerate}
        \item $f_i\mid f$ for all $1 \leq i \leq k$,
        \item $G$ is an induced subgraph of $G_{f}(D)$ where $D = \{f_1, f_2, \ldots, f_k\}$?
    \end{enumerate}
In general, the answer is no. There seem to be some subtle constraints. We discuss here a particular one. Let $g = \lcm(f_1, f_2, \ldots, f_k)$ and assume further that $g \neq f_i$ for all $1 \leq i \leq k$. If such $f$ exists, then $g \mid f$ and there is a canonical map 
\[ \Phi\colon \F_q[x]/(f) \to \F_q[x]/(g).\]

If we take a subset $S \subset \F_q[x]/(f)$ such that $|S|>|\F_q[x]/(g)|$, then there exists $a,b \in S$ such that $a \neq b$ and $\Phi(a)=\Phi(b).$ Consequently, $\Phi(a-b)=0$ or $g\mid a-b.$ By definition, $(a,b) \not \in E(G_{f}(D)).$ Consequently, we have the following upper bound for the clique number of $G_{f}(D)$ 

\[ \omega(G_{f}(D)) \leq |\F_q[x]/(g)|.\]
This shows that if $\omega(G) > |\F_q[x]/(g)| $ then $G$ cannot be an induced subgraph of $G_{f}(D).$

We can overcome the constraint by enlarging the base field $\F_q.$ 

\begin{prop}
    Let $G$ be a fixed graph. Let $f_1, f_2, \ldots, f_k$ be fixed polynomials in $\F_q[x]$. Then, there exist a finite extension $\F_{q^m}$ of $\F_q$ and a polynomial $f \in \F_{q^m}[x]$  such that 
    \begin{enumerate}
        \item $f_i \mid f$ for all $1 \leq i \leq k$,
        \item $G$ is an induced subgraph of $G_{f}(D)$ where $D = \{f_1, f_2, \ldots, f_k\}$ and we consider $f$ as an element of $\F_{q^m}[x].$
    \end{enumerate}
\end{prop}

\begin{proof}
By Erdős and Evans's theorem, there exists a squarefree number $n$ such that $G$ is an induced subgraph of the unitary Cayley graph $G_n.$ Suppose that $r = \omega(n)$ and $n=p_1 p_2 \ldots p_r$ be the prime factorization of $n.$ Then, as explained in the proof of \cref{prop:induced_unitary}, $G_n \cong \prod_{i=1}^r K_{p_i}.$ Our goal is to show that we can find $\F_{q^m}$ and $f$ such that $G_n$ is an induced subgraph of $G_{f}(D)$ where we consider $f$ as an element in $\F_{q^m}[x].$

Let $m$ be a positive integer such that $q^m>n$ and $g = \lcm(f_1, f_2, \ldots, f_k) \in \F_q[x].$ By Galois theory, there exists a polynomial $h \in \F_q[x]$ with at least $r$ distinct irreducible factors and $\gcd(h,g)=1.$ Let $f=hg$. Let $I$ be the ideal in $\F_{q^m}[x]/(f)$ generated by $f_1.$ Then, by \cref{lem:induced_graph_ideal}, the induced subgraph on $I$ is isomorphic to the unitary Cayley graph $G_{\bar{f}}(\{1\})$ where $\bar{f} = f/f_1.$ Suppose that $\bar{f} = g_1^{a_1} g_2^{a_2} \ldots g_t^{a_t}$ be the factorization of $f/f_1$ over $\F_{q^m}[x].$ Then, by the choice of $f$, $t \geq r.$ By the Chinese remainder theorem, we know that $\F_{q^m}^t$ is a subring of $\F_{q^m}[x]/\bar{f}.$ Therefore, $G_{\F_{q^m}^t} = \prod_{i=1}^t K_{q^{m}}$ is an induced subgraph of $G_{\bar{f}}(\{1\})$. Since $t \geq r$ and $q^m \geq n$, by \cref{lem:add_complete_graph} $G_n$ is an induced subgraph of  $G_{\F_{q^m}^t}$. Consequently, $G_n$ is an induced subgraph of $G_{f}(D)$ as well. 
\end{proof}

\section*{Acknowledgements}
We thank Professor Torsten Sander for the enlightening correspondence about gcd-graphs and related arithmetic graphs. We are grateful to Professor Emmanuel Kolwaski for some helpful comments on Fekete polynomials and for writing the book \cite{kowalski2018exponential}.  We thank Professor Ki-Bong Nam for his interest in this topic and his feedback on the first draft. Finally, we thank Minh-Tam Trinh for his discussions, encouragement, and support at the initial stage of this work. {We are also grateful to the referee for their comments and valuable suggestions which we have used to improve our exposition.}


\begin{thebibliography}{10}

\bibitem{unitary}
R.~Akhtar, M.~Boggess, T.~Jackson-Henderson, I.~Jim{\'e}nez, R.~Karpman, A.~Kinzel, and D.~Pritikin, \emph{On the unitary {Cayley} graph of a finite ring}, Electron. J. Combin. \textbf{16} (2009), no.~1, Research Paper 117, 13 pages.

\bibitem{artin2}
E.~Artin, \emph{Quadratische {K}{\"o}rper im {Gebiete} der h{\"o}heren {Kongruenzen}. {II}.}, Math. Z. \textbf{19} (1924), no.~1, 207--246. \MR{1544652.}

\bibitem{bipartite}
Armen~S. Asratian, Tristan M.~J. Denley, and Roland H{\"a}ggkvist, \emph{Bipartite graphs and their applications}, vol. 131, Cambridge university press, 1998.

\bibitem{bavsic2015polynomials}
Milan Ba{\v{s}}i{\'c} and Aleksandar Ili{\'c}, \emph{Polynomials of unitary {Cayley} graphs}, Filomat \textbf{29} (2015), no.~9, 2079--2086.

\bibitem{biswas2019cheeger}
Arindam Biswas, \emph{On a {Cheeger} type inequality in {Cayley} graphs of finite groups}, European Journal of Combinatorics \textbf{81} (2019), 298--308.

\bibitem{boccaletti2014structure}
Stefano Boccaletti, Ginestra Bianconi, Regino Criado, Charo~I Del~Genio, Jes{\'u}s G{\'o}mez-Gardenes, Miguel Romance, Irene Sendina-Nadal, Zhen Wang, and Massimiliano Zanin, \emph{The structure and dynamics of multilayer networks}, Physics reports \textbf{544} (2014), no.~1, 1--122.

\bibitem{brandstadt1999graph}
Andreas Brandst{\"a}dt, Van~Bang Le, and Jeremy~P Spinrad, \emph{Graph classes: a survey}, SIAM Monographs on Discrete Mathematics and Applications, 1999.

\bibitem{chidambaram2023fekete}
Shiva Chidambaram, J{\'a}n Min\'a\v{c}, Tung~T. Nguyen, and Nguyen~Duy T{a}n, \emph{{Fekete} polynomials of principal {Dirichlet} characters}, The Journal of Experimental Mathematics \textbf{1} (2025), no.~1, 51--93.

\bibitem{chudnovsky2024prime}
Maria Chudnovsky, Michal Cizek, Logan Crew, J{\'a}n Min{\'a}{\v{c}}, Tung~T. Nguyen, Sophie Spirkl, and Nguy{\^e}n~Duy T{\^a}n, \emph{On prime {Cayley} graphs}, arXiv:2401.06062, to appear in Journal of Combinatorics (2024).

\bibitem{chvatal1984topics}
Va{\^s}ek Chv{\'a}tal and Claude Berge, \emph{Topics on perfect graphs}, Elsevier, 1984.

\bibitem{dilworth1987decomposition}
Robert~P. Dilworth, \emph{A decomposition theorem for partially ordered sets}, Classic papers in combinatorics (1987), 139--144.

\bibitem{erdos1989representations}
Paul Erd{\"o}s and Anthony~B. Evans, \emph{Representations of graphs and orthogonal {Latin} square graphs}, Journal of Graph Theory \textbf{13} (1989), no.~5, 593--595.

\bibitem{gallai1967transitiv}
Tibor Gallai, \emph{Transitiv orientierbare graphen}, Acta Mathematica Hungarica \textbf{18} (1967), no.~1-2, 25--66.

\bibitem{jain2023composed}
Priya~B. Jain, Tung~T. Nguyen, J{\'a}n Min{\'a}{\v{c}}, Lyle~E. Muller, and Roberto~C. Budzinski, \emph{Composed solutions of synchronized patterns in multiplex networks of {Kuramoto} oscillators}, Chaos: An Interdisciplinary Journal of Nonlinear Science \textbf{33} (2023), no.~10.

\bibitem{kanemitsu2013matrices}
Shigeru Kanemitsu and Michel Waldschmidt, \emph{Matrices of finite abelian groups, finite {Fourier} transform and codes}, Proc. 6th China-Japan Sem. Number Theory, World Sci. London-Singapore-New Jersey (2013), 90--106.

\bibitem{kiani2012unitary}
Dariush Kiani and Mohsen Molla~Haji Aghaei, \emph{On the unitary {Cayley} graph of a ring}, Electron. J. Comb. (2012), P10--P10.

\bibitem{kiani2015unitary}
Dariush Kiani and Mohsen Mollahajiaghaei, \emph{On the unitary {Cayley} graphs of matrix algebras}, Linear Algebra and its Applications \textbf{466} (2015), 421--428.

\bibitem{kivela2014multilayer}
M.~Kivel{\"a}, A.~Arenas, M.~Barthelemy, J.~P. Gleeson, Y.~Moreno, and M.~A. Porter, \emph{Multilayer networks}, Journal of Complex Networks \textbf{2} (2014), no.~3, 203--271.

\bibitem{klotz2007some}
Walter Klotz and Torsten Sander, \emph{Some properties of unitary {Cayley} graphs}, The Electronic Journal of Combinatorics \textbf{14} (2007), no.~1, R45, 12 pages.

\bibitem{kowalski2018exponential}
E~Kowalski, \emph{Exponential sums over finite fields, {I}: elementary methods}, preparation; available at www. math. ethz. ch/\~{} kowalski/exp-sums. pdf (2018).

\bibitem{lam2012lectures}
T.~Y. Lam, \emph{Lectures on modules and rings}, Graduate Texts in Mathematics, vol. 189, Springer-Verlag, New York, 1999.

\bibitem{lamprecht1953allgemeine}
Erich Lamprecht, \emph{Allgemeine theorie der {Gau{\ss}schen} {Summen} in endlichen kommutativen {Ringen}}, Mathematische Nachrichten \textbf{9} (1953), no.~3, 149--196.

\bibitem{divisor_graphs}
Jonathan~L Merzel, J{\'a}n Min{\'a}{\v{c}}, Tung~T Nguyen, and Nguyen~Duy T{\^a}n, \emph{On divisibility relation graphs}, arXiv preprint arXiv:2507.06873 (2025).

\bibitem{Nguyen_gcd_graph}
J{\'a}n Min\'a\v{c}, Tung~T. Nguyen, and Nguyen~Duy T{a}n, \emph{Gcd-graphs over function fields}, \url{https://github.com/tungprime/gcd_graphs}, 2024.

\bibitem{mirsky1971dual}
Leon Mirsky, \emph{A dual of {Dilworth}'s decomposition theorem}, The American Mathematical Monthly \textbf{78} (1971), no.~8, 876--877.

\bibitem{nguyen2022broadcasting}
Tung~T. Nguyen, Roberto~C Budzinski, Federico~W Pasini, Robin Delabays, J{\'a}n Min{\'a}{\v{c}}, and Lyle~E Muller, \emph{Broadcasting solutions on networked systems of phase oscillators}, Chaos, Solitons \& Fractals \textbf{168} (2023), 113166.

\bibitem{nguyen2024certain}
Tung~T. Nguyen and Nguyen~Duy T{\^a}n, \emph{On certain properties of the $ p $-unitary {Cayley} graph over a finite ring}, arXiv preprint arXiv:2403.05635 (2024).

\bibitem{nguyen2024integral}
Tung~T Nguyen and Nguyen~Duy T{\^a}n, \emph{Integral {Cayley} graphs over a finite symmetric algebra}, Archiv der Mathematik \textbf{124} (2025), 615--623.

\bibitem{podesta2021_finitefield}
Ricardo~A. Podest{\'a} and Denis~E. Videla, \emph{The {Waring}’s problem over finite fields through generalized {Paley} graphs}, Discrete Mathematics \textbf{344} (2021), no.~5, 112324.

\bibitem{podesta2025GP}
\bysame, \emph{Connected components and non-bipartiteness of generalized {P}aley graphs}, Annals of Combinatorics (2025), 1--25.

\bibitem{podesta2019spectral}
\bysame, \emph{Spectral properties of generalized {Paley} graphs and their associated irreducible cyclic codes}, Australasian Journal of Combinatorics \textbf{91} (2025), no.~3, 326--365.

\bibitem{so2006integral}
Wasin So, \emph{Integral circulant graphs}, Discrete Mathematics \textbf{306} (2006), no.~1, 153--158.

\bibitem{weil1939analogie}
Andr{\'e} Weil, \emph{Sur l’analogie entre les corps de nombres alg{\'e}briques et les corps de fonctions alg{\'e}briques}, Revue Scient \textbf{77} (1939), 104--106.

\end{thebibliography}

\end{document}